\documentclass[oneside,english]{amsart}
\usepackage[T1]{fontenc}
\usepackage[latin9]{inputenc}
\usepackage{color}
\usepackage{babel}
\usepackage{verbatim}
\usepackage{refstyle}
\usepackage{units}
\usepackage{enumitem}
\usepackage{amsbsy}
\usepackage{amstext}
\usepackage{amsthm}
\usepackage{amssymb}
\usepackage{esint}
\usepackage[unicode=true,pdfusetitle,
 bookmarks=true,bookmarksnumbered=false,bookmarksopen=false,
 breaklinks=false,pdfborder={0 0 1},backref=false,colorlinks=true]
 {hyperref}

\makeatletter

\AtBeginDocument{\providecommand\thmref[1]{\ref{thm:#1}}}
\AtBeginDocument{\providecommand\propref[1]{\ref{prop:#1}}}
\AtBeginDocument{\providecommand\lemref[1]{\ref{lem:#1}}}
\newcommand{\lyxmathsym}[1]{\ifmmode\begingroup\def\b@ld{bold}
  \text{\ifx\math@version\b@ld\bfseries\fi#1}\endgroup\else#1\fi}

\RS@ifundefined{subsecref}
  {\newref{subsec}{name = \RSsectxt}}
  {}
\RS@ifundefined{thmref}
  {\def\RSthmtxt{theorem~}\newref{thm}{name = \RSthmtxt}}
  {}
\RS@ifundefined{lemref}
  {\def\RSlemtxt{lemma~}\newref{lem}{name = \RSlemtxt}}
  {}

\numberwithin{equation}{section}
\numberwithin{figure}{section}
  \theoremstyle{plain}
  \newtheorem{lem}{\protect\lemmaname}
  \theoremstyle{plain}
  \newtheorem{prop}{\protect\propositionname}
  \theoremstyle{remark}
  \newtheorem{rem}{\protect\remarkname}
  \theoremstyle{plain}
  \newtheorem{assumption}{\protect\assumptionname}
\theoremstyle{plain}
\newtheorem{thm}{\protect\theoremname}
  \theoremstyle{remark}
  \newtheorem*{acknowledgement*}{\protect\acknowledgementname}

\def\RSthmtxt{Theorem~}
\def\RSlemtxt{Lemma~}

\AtBeginDocument{\providecommand\propref[1]{\ref{prop:#1}}}
\RS@ifundefined{propref}{\newref{prop}{name=Proposition~,names=propositions~}}{}

\AtBeginDocument{}
\RS@ifundefined{assuref}{\newref{assu}{name=Assumption~,names=assumptions~}}{}

\IfFileExists{lmodern.sty}{\usepackage{lmodern}}{}

\makeatother

  \providecommand{\acknowledgementname}{Acknowledgement}
  \providecommand{\assumptionname}{Assumption}
  \providecommand{\lemmaname}{Lemma}
  \providecommand{\propositionname}{Proposition}
  \providecommand{\remarkname}{Remark}
\providecommand{\theoremname}{Theorem}

\begin{document}
\global\long\def\integral#1#2{{\displaystyle \intop_{#1}^{#2}}}

\title[CTRW as a RWRE]{continuous time random walk as a random walk in a random environment}

\author{ofer busani}
\address{bar ilan university\\ ramat gan\\ israel}
\begin{abstract}
We show that for a weakly dense subset of the domain of attraction
of a positive stable random variable of index $0<\alpha<1$($DOA\left(\alpha\right))$
the functional stable convergence is a time-changed renewal convergence
of distribution of finite mean. Applied to Continuous Time Random
Walk(CTRW) \'{a} la Montroll and Wiess we show that CTRW with renewal
times in a weakly dense set of $DOA\left(\alpha\right)$ can be realized
as random walk in a random environment. We find the quenched limit
and give a bound on the error of the approximation.
\end{abstract}

\maketitle

\section{Introduction}

Let $\left\{ W_{i}\right\} _{i=1}^{\infty}$ (abbrv. $\left\{ W_{i}\right\} $)
be a sequence of i.i.d positive r.vs s.t $\mathbb{P}\left(W_{1}>t\right)\sim t^{-\alpha}$
for $0<\alpha<1$. Then it is well known that the process $D_{t}^{n}=n^{-\frac{1}{\alpha}}\sum_{i=1}^{\left[nt\right]}W_{i},$
converges weakly in the $J_{1}$ topology to a stable subordinator,
that is
\begin{equation}
D^{n}\overset{J_{1}}{\Rightarrow}D,\label{eq:convergence to subordinator}
\end{equation}
where $\overset{J_{1}}{\Rightarrow}$ denotes weak convergence w.r.t
to $J_{1}$-Skorohod topology. The fact that $W_{1}$ typically has
big jumps carries over to the limit. This is in contrast to the SLLN
of the Renewal Theorem that says that if $\left\{ U_{i}\right\} $
is a sequence of i.i.d r.vs s.t $\mathbb{E}\left(U_{1}\right)=1$
then $T^{n}=n^{-1}\sum_{i=1}^{\left[nt\right]}U_{i}$ converges in
the Skorohod topology to the function $t\mapsto t$, i.e
\begin{equation}
T^{n}\overset{J_{1}}{\rightarrow}t,\label{eq:convergence to identity}
\end{equation}
where $\overset{J_{1}}{\rightarrow}$ denotes a.s convergence w.r.t
$J_{1}$ topology. We wish to show here that these two apparently
different convergences are closely related. That in fact, observing
the convergence in (\ref{eq:convergence to subordinator}) is essentially
observing the convergence in (\ref{eq:convergence to identity}) viewed
through a sequence of random embedding of the positive real line into
itself. One use of the convergence in (\ref{eq:convergence to subordinator})
is in the model of Continuous Time Random Walks (CTRW) introduced
in \cite{Montroll1965} by Montroll and Wiess. In the most simple
setup $\left\{ J_{i}\right\} $ and $\left\{ W_{i}\right\} $ are
two independent sequences of i.i.d r.vs. Define $\left(S_{n},T_{n}\right)=\left(\sum_{i=1}^{n}J_{i},\sum_{i=1}^{n}W_{i}\right)$,
the (uncoupled) CTRW associated with space-time jumps $\left\{ \left(J_{i},W_{i}\right)\right\} _{i=1}^{\infty}$
(abbrv. $\left(J_{i},W_{i}\right)$) is 
\[
X_{t}=\sum_{i=1}^{N_{t}}J_{i},
\]
where $N_{t}=\sup\left\{ n:T_{n}\leq t\right\} $. In order to model
the microscopic behavior of a particle with long binding times to
a substrate, one assumes that $W_{1}$ is heavy tailed, that is 
\[
\mathbb{P}\left(W_{1}>t\right)\sim t^{-\alpha},
\]
for some $0<\alpha<1$. The functional limit of $X_{t}$ for large
$t$ was first considered in  \cite{Meerschaert2004} in the mathematics
literature although earlier in the physics literature (\cite{Barkai2000}).
Limits for coupled CTRW were considered in \cite{becker2004limit},
and in \cite{Meerschaert2008} that of CTRW with space-time jumps
that are infinitely divisible. It was shown that 
\begin{equation}
n^{-1}X_{tn^{\frac{2}{\alpha}}}\overset{J_{1}}{\Rightarrow}B_{E_{t}},\label{eq:convergence of CTRW}
\end{equation}
where $B_{t}$ is a Brownian motion and $E_{t}$ is the inverse stable
subordinator independent of $B_{t}$ defined by 
\[
E_{t}=\inf\left\{ s:D_{s}>t\right\} .
\]
The process $B_{E_{t}}$, sometimes called the Fractional Kinetics
process, is a sub-diffusion in the since that it is self-similar with
exponent $\frac{\alpha}{2}$, i.e
\[
B_{E_{tc}}\sim c^{\frac{\alpha}{2}}B_{E_{t}}.
\]
Our results show that the invariance principle in (\ref{eq:convergence of CTRW})
where the limit is a Bm subordinated to an independent inverse subordinator,
is not merely a property of the limit but is the case for the CTRW
itself, even when the CTRW is coupled, i.e. when the r.v $W_{i}$
and $J_{i}$ are dependent. In fact, we show this for a larger set
of CTRWs, namely CTRW with waiting times with infinite mean with some
restriction on their Laplace Transform. A simple case is when $X_{t}$
is an uncoupled CTRW associated with the i.i.d space-time jumps $\left(J_{i},W_{i}\right)$,
where $W_{i}\in DOA\left(\alpha\right)$. Then we show that for every
$\epsilon>0$ one can construct a probability space where one can
find a sequence of i.i.d r.vs $\left(J_{i},U_{i}\right)$ where $\mathbb{E}\left(U_{1}\right)<\infty$
and an inverse subordinator (not necessarily stable) $E_{t}$, independent
of $\left\{ U_{i}\right\} $, s.t if $Y_{t}$ is the CTRW associated
with $\left(J_{i},U_{i}\right)$ then
\begin{equation}
\rho_{d_{J_{1}}}\left(Y_{E_{t}},X_{t}\right)<\epsilon,\label{eq:distance of laws}
\end{equation}
where $\rho_{d_{J_{1}}}$ is the Prohorov metric on probability distributions
metrizing the weak topology of distributions on the Skorohod space
$\mathbb{D}([0,\infty))$. This enables us to show that by enriching
the filtration of a CTRW one may realize CTRW as an annealed process
of a random walk in a random environment (RWRE). One of are two main
results (\thmref{CTRW_as_RWRE}) shows that there exists a set of
distributions $\mathcal{A}$ which is weakly dense in $DOA\left(\alpha\right)$
for which CTRW is an annealed process of RWRE. The random environment
is a random time change while the quenched process is a CTRW with
finite mean waiting times (independent of the enviornment) time-changed
by the random enviornment. The results also show that there exists
a set of distributions $\mathcal{B}\subset\mathcal{A}$ which can
be realized as another RWRE. This time the random environment is traps
in time, that is, for each time $n\in\mathbb{Z}_{+}$ one randomizes
i.i.d trappings $\tau_{n}$ from a heavy tailed distribution, the
quenched process will then be a CTRW with waiting times $\left\{ \tau_{i}U_{i}\right\} $,
where $\mathbb{E}\left(U_{1}\right)<\infty$. We also show that under
proper scaling of CTRW, the quenched process converges to an interesting
diffusion time changed by the inverse of a stable subordinator. It
shows that in the quenched limit the dynamics of the space-time jumps
$\left(J_{i},U_{i}\right)$ are translated to that of the regenerative
points of the enviornment. Our second main result (\thmref{quantitive_result})
deals with trying to bound the distance in (\ref{eq:distance of laws})
when we scale the process' $Y_{E_{t}}$ and $X_{t}$ . We give a polynomial
bound $Cn^{-c}$, however, the proof gives way to finding a better
$c$ if one only finds a good way of matching the tail of a subordinator
with that of $W_{1}$. Note that CTRW were considered in \cite{Arous2015}
as one instance of a RWRE on $\mathbb{Z}$ called \emph{a Randomly
Trapped Random Walk} (RTRW).\emph{ }However, there, the random environment
is probability measures $\left\{ \pi_{z}\left(dt\right)\right\} _{z\in\mathbb{Z}}$
on the the positive real line. Given such a random environment, one
preforms a simple random walk on $\mathbb{Z}$ with waiting times
$\left\{ W_{i}^{z}\right\} _{i\in\mathbb{Z}_{+},z\in\mathbb{Z}}$
s.t the sequence $\left\{ W_{i}^{z}\right\} _{i\in\mathbb{Z}_{+}}$
of waiting times at site $z$ is drawn independently from the the
distribution $\pi_{z}$. Reaching the site $x\in\mathbb{Z}$ for the
$i$'th time, the random walk waits $W_{i}^{x}$ before moving on
to the next site, i.e. traps are in space. In contrast, we show that
CTRWs can, at some instances (e.g. stable distribution, Mittag-Leffler
distribution), be realized as trap models where the traps are in time
rather than space. Moreover, presented as a RTRW, CTRWs are essentially
degenerate in the sense that the environment is deterministic, and
therefore the limit is completely annealed. Here we show that by considering
a larger filtration, the quenched limit retains its environment. 

\section{\label{sec:Prelimineries}Preliminaries}

Recall that a Bernstein function is a function $f:\left(0,\infty\right)\rightarrow\mathbb{R}$
that is infinitely differentiable, $f\left(s\right)\geq0$ and $\left(-1\right)^{n-1}f^{\left(n\right)}\left(s\right)\geq0$
for $n\geq1$, where $f^{\left(n\right)}$ is the n'th derivative
of $f$. A function $f$ is a Bernstein function iff $f$ is of the
form 
\begin{equation}
f\left(s\right)=a+bs+\integral 0{\infty}\left(1-e^{-sy}\right)\mu\left(dy\right),\label{eq:Psi}
\end{equation}
where $a,b\geq0$ and $\mu$ is a measure on $\left(0,\infty\right)$
s.t $\integral 0{\infty}\left(1\wedge y\right)\mu\left(dy\right)<\infty$,
one can then identify a Bernstein function with the characteristics
$\left(a,b,\mu\right)$. We shall be interested in the set
\[
\mathfrak{B}:=\left\{ f:f\text{ is an unbounded Bernstein function of characteristics }\left(0,b,\mu\right)\right\} .
\]
We denote the Laplace Transform(LT) of a positive measure $\mu$ on
$\left(0,\infty\right)$ by $\mathcal{L}\mu\left(s\right)=\integral 0{\infty}e^{-st}\mu\left(dt\right)$.
Let $\mathcal{CM}$ denote the space of completely monotone functions,
i.e., $f\in\mathcal{CM}$ iff $f:\left(0,\infty\right)\rightarrow\mathbb{R}$
and $\left(-1\right)^{n}f^{\left(n\right)}\left(s\right)\geq0$ for
$n\geq0$. Define
\[
\mathfrak{\hat{L}}:=\left\{ f\in\mathcal{CM}:f\left(0^{+}\right)=1\right\} .
\]
 Recall that $\mathfrak{\hat{L}}$ is just the set of Laplace Transforms
of probability measures on the positive real line. For $\psi\in\mathfrak{B}$
we define the mapping $\hat{\Phi}_{\psi}:\mathfrak{\mathfrak{\hat{L}}}\rightarrow\mathfrak{\mathfrak{\hat{L}}}$
by 
\[
\hat{\Phi}_{\psi}\left(f\right)\left(s\right)=f\left(\psi\left(s\right)\right).
\]
Note that the mapping is indeed into $\mathfrak{\mathfrak{\hat{L}}}$;
if $f\in\mathcal{CM}$ and $\psi$ is a Bernstein function then $f\left(\psi\left(s\right)\right)\in\mathcal{CM}$(\cite[Theorem 3.6]{schilling2012bernstein}).
We say the function $L:\mathbb{R}^{+}\rightarrow\mathbb{R}$ is slowly
varying at $\infty$ if 
\begin{equation}
\lim_{x\rightarrow\infty}\frac{L\left(\lambda x\right)}{L\left(x\right)}=1,\label{eq:SV function}
\end{equation}
for every $\lambda\in\mathbb{R}^{+}$. In fact it is enough to show
that (\ref{eq:SV function}) holds for every $\lambda\in\Lambda$,
where $\Lambda\subset\mathbb{R}$ is of positive Lebesgue measure.
Next, define 
\[
\mathcal{\mathfrak{\hat{L}}_{\psi}}:=\left\{ f\in\mathfrak{\mathfrak{\hat{L}}}:f\sim1-\psi\left(s\right)L\left(s^{-1}\right)\right\} ,
\]
where $L$ is a slowly varying function, and where $f\sim1-\psi\left(s\right)L\left(s^{-1}\right)$
means that 
\[
\lim_{s\rightarrow0}\frac{\left|f\left(s\right)-1\right|}{\psi\left(s\right)L\left(s^{-1}\right)}=1.
\]
 We also use $X\sim f$, where $X$ is a r.v and $f$ is a distribution
or a r.v, to say that $X$ is distributed according to $f$, there
should not be a confusion there. We denote by $\mathfrak{L}$ and
$\mathfrak{L}_{\psi}$ the space of measures whose Laplace transform(LT)
is in $\mathfrak{\mathfrak{\hat{L}}}$ and $\mathcal{\mathfrak{\hat{L}}_{\psi}}$
respectively, that is, the LT $\mathcal{L}$ is a bijection between
$\mathfrak{L}$ and $\mathfrak{\mathfrak{\hat{L}}}$ and between $\mathfrak{L}_{\psi}$
and $\mathcal{\mathfrak{\hat{L}}_{\psi}}$. We also define $\Phi_{\psi}:\mathfrak{L}\rightarrow\mathfrak{L}$
as $\Phi_{\psi}:=\mathcal{L}^{-1}\hat{\Phi}_{\psi}\mathcal{L}$. If
$X$ is a r.v with distribution $f$, we sometimes write $\Phi_{\psi}\left(X\right)$
instead of $\Phi_{\psi}\left(f\right)$. Finally, let $\psi_{1}$
and $\psi_{2}$ be two Bernstein functions in $\mathfrak{B}$, we
define $\hat{\mathfrak{L}}_{\psi_{1}}^{\psi_{2}}=\hat{\Phi}_{\psi_{2}}\left(\hat{\mathfrak{L}}_{\psi_{1}}\right)$
and $\mathfrak{L}_{\psi_{1}}^{\psi_{2}}=\Phi_{\psi_{2}}\left(\mathfrak{L}_{\psi_{1}}\right)$.

Recall that a positive r.v $X$ is said to be in the domain of attraction
of a stable (totally asymmetric) r.v $Y$ of index $0<\alpha<1$,
i.e. $\mathbb{E}\left(e^{-sY}\right)=e^{-s^{\alpha}}$ (abbr. $X\in DOA\left(\alpha\right)$),
if there exists a sequence of normalizing constants $a_{n}\rightarrow0$
s.t 
\[
a_{n}\sum_{i=1}^{n}X_{i}\Rightarrow Y,
\]
where $\left\{ X_{i}\right\} $ are i.i.d copies of $X$ and $\Rightarrow$
denotes weak convergence of measures. It is well known that $X\in DOA\left(\alpha\right)$
iff $\mathbb{P}\left(X>t\right)\sim L\left(t\right)t^{-\alpha}$,
where $L\left(t\right)$ is a slowly varying function. It is also
known that the sequence $a_{n}$ is regularly varying, i.e, 
\[
\lim_{n\rightarrow\infty}\frac{a_{\left[\lambda n\right]}}{a_{n}}=\lambda^{-\alpha}\qquad\lambda>0,
\]
and that 
\begin{equation}
nL\left(a_{n}^{-1}t\right)\left(a_{n}^{-1}t\right)^{-\alpha}\rightarrow\frac{t^{-\alpha}}{\Gamma\left(1-\alpha\right)}.\label{eq:the sequence a_n}
\end{equation}
For convenience we let $a_{1}=1$. Our interest in Bernstein functions
and the mappings $\Phi_{\psi}$ is in part due to the following fact:
for $0<\alpha<1$ $\mathfrak{L}_{s^{\alpha}}=DOA\left(\alpha\right)$.
Moreover, 
\begin{equation}
\mathfrak{L}_{s}=\left\{ \mu:\mu\in\mathfrak{L},\integral 0ty\mu\left(dy\right)\quad\text{is slowly varying}\right\} .\label{eq:characterization of the first layer}
\end{equation}
These are consequences of \cite[Corollary 8.1.7 and Theorem 8.3.1]{bingham1989regular}.\\
Let $f:\mathbb{R}\rightarrow\mathbb{R}$ be a right continuous function
with left limits. We denote 
\[
f_{t-}:=\lim_{\epsilon\rightarrow0^{+}}f\left(t-\epsilon\right),
\]
the left limit of $f_{t}$. If $f$ is a left continuous function
with right limits then 
\[
\left(f_{t}\right)^{+}:=\lim_{\epsilon\rightarrow0^{+}}f\left(t+\epsilon\right),
\]
the right limit of $f_{t}$. Note that whenever $g\left(t\right)$
is a continuous strictly increasing function and $f$ is right continuous
with left limits $\left(f_{g\left(t\right)-}\right)^{+}=f_{g\left(t\right)}$
(note that we first compute $f_{t-}$ and then evaluate at $g\left(t\right)$).
This may not be the case when there exists $\epsilon>0$ s.t $g\left(t-\epsilon\right)=g\left(t+\epsilon\right)$
and $f$ is not continuous at $g\left(t\right)$ . We say that $X_{t}$
is a CTRW with space-time jumps $\left\{ J_{i},W_{i}\right\} $ or
that $X_{t}$ is a CTRW associated with the space-time jumps $\left\{ J_{i},W_{i}\right\} $
, if 
\[
X_{t}=\sum_{n=1}^{\infty}J_{n}1\left(t\right)_{\left\{ y:T_{n}\leq y\right\} },
\]
where $T_{n}=\sum_{i=1}^{n}W_{i}$. We use $\mathbb{D}[0,T]$ ($\mathbb{D}[0,\infty)$)
to denote the subspace of $\mathbb{R}^{[0,T]}$ ($\mathbb{R}^{\mathbb{R}_{+}})$for
$T>0$ of c�dl�g functions, and $\overset{J_{1}[0,T]}{\sim}$ to denote
the equivalence of law of processes in the Skorohod $J_{1}$ topology
on $\mathbb{D}[0,T]$ . We shall use $\mathbb{D}$ and $\overset{J_{1}}{\sim}$
when we refer to $\mathbb{D}[0,\infty)$ and $\overset{J_{1}}{\sim}$.
We use $X_{t}^{n}\overset{J_{1}}{\Rightarrow}X_{t}$ ($X_{t}^{n}\overset{J_{1}[0,T]}{\Rightarrow}X_{t}$)
to say that the law of the process $X_{t}^{n}$ converges weakly to
that of $X_{t}$ w.r.t the $J_{1}$ topology on $\mathbb{D}$ ($\mathbb{D}[0,T]$).
Let $d$ be a metric on the set $V$ and let $\mathcal{P}\left(V\right)$
be the set of all probability measures on the Borel sets (with respect
to $d$) of $V$. Recall that a sequence of probability measures $p_{n}\in\mathcal{P}\left(V\right)$
converges weakly to $p\in\mathcal{P}\left(V\right)$ if for every
bounded continuous (with respect to $d$) function $h:V\rightarrow\mathbb{R}$
we have
\[
\int h\left(x\right)p_{n}\left(dx\right)\rightarrow\int h\left(x\right)p\left(dx\right).
\]
Recall further that the weak topology of $\mathcal{P}\left(V\right)$
is metrizable by the following metric
\[
\rho_{d}\left(p_{1},p_{2}\right)=\inf_{p_{1,2}}\inf\left\{ \epsilon:p_{1,2}\left(\left|X-Y\right|>\epsilon\right)<\epsilon\right\} ,
\]
where the infimum runs over all couplings of the r.vs $X$ and $Y$
whose distribution is given by $p_{1}$ and $p_{2}$ respectively.
For two r.vs $X$ and $Y$ we sometimes write $\rho_{d}\left(X,Y\right)$,
which should be understood as $\rho\left(p_{X},p_{Y}\right)$, where
$p_{X}$ and $p_{Y}$ are the distributions of $X$ and $Y$ respectively.
Recall that the Skorohord $J_{1}$ topology on $\mathbb{D}[0,T]$
is metrizable in the following way; a sequence $f_{t}^{n}\in\mathbb{R}^{[0,T]}$
converges in the $J_{1}$ topology to $f_{t}\in\mathbb{R}^{[0,T]}$
if there exists a sequence of homeomorphisms $\lambda_{t}^{n}:[0,T]\rightarrow[0,T]$
s.t
\[
\left\Vert f_{\lambda^{n}}^{n}-f\right\Vert \rightarrow0\quad\textrm{and}\quad\left\Vert \lambda_{t}^{n}-t\right\Vert \rightarrow0,
\]
as $n\rightarrow\infty$, where $\left\Vert \cdot\right\Vert $ is
the sup norm, that is, for $f_{t},g_{t}\in\mathbb{R}^{[0,T]}$ 
\[
\left\Vert g-f\right\Vert =\sup_{t\in[0,T]}\left|g_{t}-f_{t}\right|.
\]
Denote by $\Lambda$ the set of all homeomorphismes form $[0,T]$
to itself. One way to metrize the $J_{1}$ topology is to use the
following metric
\[
d_{J_{1}}\left(f,g\right)=\inf_{\lambda\in\Lambda}\left\{ \left\Vert g_{\lambda}-f\right\Vert \lor\left\Vert \lambda_{t}-t\right\Vert \right\} .
\]
Let $\mathbb{D}_{\uparrow\uparrow}$ be the subset in $\mathbb{D}$
whose elements are strictly increasing. If $d_{t}\in\mathbb{D}$ and
increasing, we define the generalized inverse of $d_{t}$ to be 
\[
d_{t}^{-1}=\inf\left\{ s:d_{s}>t\right\} .
\]
 Note that $d_{t}^{-1}$ is continuous iff $d_{t}$ is strictly increasing.
Define the mapping $\mathcal{H}:\mathbb{D}\times\mathbb{D}_{\uparrow\uparrow}\rightarrow\mathbb{D}$
by 
\[
\mathcal{H}\left(f_{t},d_{t}\right)=\left(f_{d_{t}^{-1}-}\right)^{+}.
\]
The results in \cite{straka2011} show that $\mathcal{H}$ is continuous
w.r.t the $J_{1}$ topology. In fact, we shall often make use of the
following result by Straka and Henry (\cite[Theorem 3.6]{straka2011}).
\\

\begin{lem}
\label{lem:Straka_and_Henry}(Straka and Henry, 2011) Suppose we have
a sequence of random space-time jumps $\left\{ J_{i}^{n},W_{i}^{n}\right\} $
and a sequence of random increasing step process $N_{t}^{n}$ s.t
\[
\left(J_{N_{t}^{n}}^{n},W_{N_{t}^{n}}^{n}\right)\overset{J_{1}}{\Rightarrow}\left(A_{t},D_{t}\right),
\]
where $D_{t}\in\mathbb{D}_{\uparrow\uparrow}$. If $X_{t}^{n}$ is
the CTRW associated with $\left\{ J_{i}^{n},W_{i}^{n}\right\} $.
Then 
\begin{equation}
X_{t}^{n}\overset{J_{1}}{\Rightarrow}\mathcal{H}\left(A_{t},D_{t}\right).\label{eq:continuity of H}
\end{equation}
\end{lem}
As in this paper we are interested mostly in the temporal jumps of
our CTRWs one may assume throughout that the spatial jumps $\left\{ J_{i}^{n}\right\} \in\mathbb{R}^{d}$
for $d\in\mathbb{N}$ are i.i.d such that 
\[
\lim_{n\rightarrow\infty}\sum_{i=1}^{\left[nt\right]}J_{i}^{n}\overset{J_{1}}{\Rightarrow}B_{t},
\]
where $B_{t}$ is a standard Bm in $\mathbb{R}^{d}$. We use the term
\emph{time-change} for a function $f$ s.t $f\left(0\right)=0$, $f$
is increasing and continuous. 

\section{from relative stability to sub-diffusion}

We begin with some technical lemmas that will be useful in understanding
the mapping $\Phi_{\psi}$.
\begin{lem}
\label{lem:Slowly Varying result}Let $L$ be a slowly varying function
and $\phi\left(s\right)$ a positive function s.t for every $\lambda>0$
there exist positive constants $C_{1}\left(\lambda\right)$ and $C_{2}\left(\lambda\right)$
s.t
\begin{equation}
C_{1}\left(\lambda\right)\leq\frac{\phi\left(\lambda s\right)}{\phi\left(s\right)}\leq C_{2}\left(\lambda\right)\qquad\forall s>S\left(\lambda\right),\label{eq:Bounding phi}
\end{equation}
for some positive constant $S\left(\lambda\right)$ that may depend
on $\lambda$. Then $L\left(\phi\left(s\right)\right)$ is again slowly
varying.
\end{lem}
\begin{proof}
Indeed, by the Uniform Convergence Theorem (UCT) (\cite[Theorem 1.2.1]{bingham1989regular})
for slowly varying functions we know that 
\[
\lim_{s\rightarrow\infty}\frac{L\left(\lambda s\right)}{L\left(s\right)}=1,
\]
uniformly on any compact $\lambda$-set in $\left(0,\infty\right)$.
Since by (\ref{eq:Bounding phi}) there exists $\lambda'\in[C_{1},C_{2}]$
s.t for every $s>S$ 
\[
\frac{L\left(\phi\left(\lambda s\right)\right)}{L\left(\phi\left(s\right)\right)}=\frac{L\left(\lambda'\phi\left(s\right)\right)}{L\left(\phi\left(s\right)\right)},
\]
taking the limit while using the uniform convergence we obtain the
result.
\end{proof}
\begin{lem}
\label{lem:Injection_result}Let $\psi_{1},\psi_{2}\in\mathfrak{B}$,
then 
\begin{equation}
\hat{\Phi}_{\psi_{2}}\left(\mathfrak{\mathfrak{\hat{L}}}_{\psi_{1}}\right)\subset\mathfrak{\mathfrak{\hat{L}}}_{\psi_{1}\left(\psi_{2}\right)}.\label{eq:Injection result : injective}
\end{equation}
\end{lem}
\begin{proof}
Suppose first that $f\in\mathfrak{\mathfrak{\hat{L}}}_{\psi_{1}}$.
By definition $f\left(s\right)\sim1-\psi_{1}\left(s\right)L\left(\frac{1}{s}\right)$
when $s\rightarrow0$ where $L$ is a slowly varying function. It
then follows that $\hat{\Phi}_{\psi_{2}}f\left(s\right)\sim1-\psi_{1}\left(\psi_{2}\left(s\right)\right)L\left(\frac{1}{\psi_{2}\left(s\right)}\right)$.
Denote $L'\left(\frac{1}{s}\right)=L\left(\frac{1}{\psi_{2}\left(s\right)}\right)$.
We must show that $L'\left(s\right)=L\left(\frac{1}{\psi_{2}\left(s^{-1}\right)}\right)$
is slowly varying. By Lemma \ref{lem:Slowly Varying result} it is
enough to show that 
\[
C_{1}\leq\frac{\psi_{2}\left(s^{-1}\right)}{\psi_{2}\left(\left(\lambda s\right)^{-1}\right)}\leq C_{2},
\]
for some positive constants $C_{1}$ and $C_{2}$ that may depend
on $\lambda$. First assume that $\psi_{2}$ has representation $\left(0,0,\mu\right)$.
From \cite[Lemma 3.4]{schilling2012bernstein} we see that
\begin{equation}
\frac{e-1}{e}\lambda\frac{I_{\mu}\left(s\right)}{I_{\mu}\left(\lambda s\right)}\leq\frac{\psi_{2}\left(s^{-1}\right)}{\psi_{2}\left(\left(\lambda s\right)^{-1}\right)}\leq\frac{e}{e-1}\lambda\frac{I_{\mu}\left(s\right)}{I_{\mu}\left(\lambda s\right)},\label{eq:Bounding psi}
\end{equation}
for every $s>0$ where $I_{\mu}\left(s\right)=\integral 0s\mu\left(y,\infty\right)dy$.
Suppose first that $\lambda\geq1$ then by the fact that $\psi_{2}$
is increasing $\frac{\psi_{2}\left(s^{-1}\right)}{\psi_{2}\left(\left(\lambda s\right)^{-1}\right)}\geq1$,
which shows that 
\[
1\leq\frac{\psi_{2}\left(s^{-1}\right)}{\psi_{2}\left(\left(\lambda s\right)^{-1}\right)}\leq\frac{e}{e-1}\lambda.
\]
Similarly, if $\lambda<1$ we have
\[
\frac{e-1}{e}\lambda\leq\frac{\psi_{2}\left(s^{-1}\right)}{\psi_{2}\left(\left(\lambda s\right)^{-1}\right)}\leq1.
\]
Now, if $\mbox{\ensuremath{\psi}}_{2}\left(s\right)=bs+\psi'\left(s\right)$,
where $\psi'\left(s\right)$ has representation $\left(0,0,\mu\right)$
and $b>0$, then
\begin{alignat*}{1}
\frac{\psi_{2}\left(s^{-1}\right)}{\psi_{2}\left(\left(\lambda s\right)^{-1}\right)} & =\frac{bs^{-1}+\psi'\left(s^{-1}\right)}{b\lambda^{-1}s^{-1}+\psi'\left(\lambda^{-1}s^{-1}\right)}\\
 & =\frac{b+\psi'\left(s^{-1}\right)s}{b\lambda^{-1}+\psi'\left(\lambda^{-1}s^{-1}\right)s}.
\end{alignat*}
 We see that for $\lambda\geq1$ , 
\[
\frac{b+\psi'\left(s^{-1}\right)s}{b\lambda^{-1}+\psi'\left(s^{-1}\right)s}\leq\frac{b+\psi'\left(s^{-1}\right)s}{b\lambda^{-1}+\psi'\left(\lambda^{-1}s^{-1}\right)s}\leq\frac{b+\psi'\left(s^{-1}\right)s}{b\lambda^{-1}+\frac{e-1}{e}\lambda\psi'\left(s^{-1}\right)s}.
\]
 Note that by integration by parts and monotone convergence we see
that the limit
\begin{align*}
M & =\lim_{s\rightarrow\infty}\psi'\left(s^{-1}\right)s\\
 & =\lim_{s\rightarrow\infty}s\integral 0{\infty}s^{-1}e^{-s^{-1}y}\mu\left(y,\infty\right)dy\\
 & =\integral 0{\infty}\mu\left(y,\infty\right)dy
\end{align*}
exists and $M\in[0,\infty]$. It follows that for some large enough
$S$, for every $s>S$ we have 
\[
C_{1}\left(\lambda\right)\leq\frac{\psi_{2}\left(s^{-1}\right)}{\psi_{2}\left(\left(\lambda s\right)^{-1}\right)}\leq C_{2}\left(\lambda\right).
\]
This shows that $\frac{1}{\psi_{2}\left(s^{-1}\right)}$ satisfies
(\ref{eq:Bounding phi}), $L\left(\frac{1}{\psi_{2}\left(s^{-1}\right)}\right)$
is slowly varying and that (\ref{eq:Injection result : injective})
holds. 
\end{proof}
We say that a measure $\mu$ is sub-homogeneous(super-homogeneous)
if for every $\lambda>0$ there exists a constant $C\left(\lambda\right)$
s.t $\mu\left(C\left(\lambda\right)x,\infty\right)\leq\lambda\mu\left(x,\infty\right)$($\lambda\mu\left(x,\infty\right)\leq\mu\left(C\left(\lambda\right)x,\infty\right)$)
for every $x>0$. For example, if $\mu\left(dx\right)$ is a finite
measure and $\mu\left(x,\infty\right)=x^{-\alpha}L\left(x\right)$
where $L\left(x\right)$ converges to a constant at infinity, then
$\mu$ is sub-homogeneous. The following is a partial uniqueness result.
\begin{lem}
\label{lem:uniqueness_result}Let $\psi_{1},\psi_{2}\in\mathfrak{B}$,
and assume that the measure $\mu_{2}$ of $\psi_{2}$ is sub-homogeneous
or super-homogeneous. Then $\Phi_{\psi_{2}}^{-1}\left(\mathcal{\mathfrak{L}}_{\psi_{1}\left(\psi_{2}\right)}\right)=\mathcal{\mathfrak{L}}_{\psi_{1}}$
. 
\end{lem}
\begin{proof}
We prove this for when $\mu_{2}$ is sub-homogeneous as the proof
for the super-homogeneous is similar. Let $f\in\hat{\mathfrak{L}}$
s.t $\hat{\Phi}_{\psi_{2}}f\in\mathcal{\mathfrak{\hat{L}}}_{\psi_{1}\left(\psi_{2}\right)}$,
or equivalently that $\hat{\Phi}_{\psi_{2}}f\sim1-\psi_{1}\left(\psi_{2}\left(s\right)\right)L\left(s^{-1}\right)$
as $s\rightarrow0$ where $L\left(s\right)$ is slowly varying. It
follows that $\hat{\Phi}_{\psi_{2}}^{-1}f\sim1-\psi_{1}\left(s\right)L\left(\frac{1}{\psi_{2}^{-1}\left(s\right)}\right)$,
and we must show that $L'\left(s\right)=L\left(\frac{1}{\psi_{2}^{-1}\left(\frac{1}{s}\right)}\right)$
is slowly varying. By the characterization of regularly varying function,
in order to show that $L'\left(s\right)$ is slowly varying it is
enough to show that 
\[
\frac{L'\left(\lambda s\right)}{L'\left(s\right)}\rightarrow1,
\]
for $\lambda\in\Lambda$ where $\Lambda$ is a set of positive measure.
Let $\lambda\in[1,\infty)$, it is then enough to show that 
\begin{equation}
C'_{1}\leq\frac{\frac{1}{\psi_{2}^{-1}\left(\frac{1}{\lambda s}\right)}}{\frac{1}{\psi_{2}^{-1}\left(\frac{1}{s}\right)}}\leq C'_{2},\label{eq:uniqueness result :  bounds}
\end{equation}
for some positive constants $C'_{1},C'_{2}$ that may depend on $\lambda$.
Since $\psi_{2}^{-1}\left(s\right)$ is increasing we see that 
\[
1\leq\frac{\psi_{2}^{-1}\left(\frac{1}{s}\right)}{\psi_{2}^{-1}\left(\lambda^{-1}\frac{1}{s}\right)}.
\]
It is now enough to show that $\psi_{2}^{-1}\left(t\right)\leq C'_{2}\left(k\right)\psi_{2}^{-1}\left(kt\right)$
for $0<k\leq1$ and positive $C'_{2}\left(k\right)$. Let $t=bs+\integral 0{\infty}\left(1-e^{-sy}\right)\mu\left(dy\right)$
so that $\psi_{2}^{-1}\left(t\right)=s$. By the fact that $\mu$
is sub-homogeneous we see that 
\begin{alignat*}{1}
kt & =kbs+k\integral 0{\infty}\left(1-e^{-sy}\right)\mu_{2}\left(dy\right)\\
 & \geq kbs+\integral 0{\infty}\left(1-e^{-sy}\right)\mu_{2}\left(C\left(k\right)dy\right)\\
 & =kbs+\integral 0{\infty}\left(1-e^{-sC\left(k\right)y}\right)\mu_{2}\left(dy\right)\\
 & \geq C'\left(k\right)bs+\integral 0{\infty}\left(1-e^{-sC'\left(k\right)y}\right)\mu_{2}\left(dy\right),
\end{alignat*}
where $C'\left(k\right)=\min\left\{ C\left(k\right),k\right\} $.
By the fact that $\psi_{2}^{-1}$ is increasing we have $\psi_{2}^{-1}\left(kt\right)\geq sC'\left(k\right)$,
or that $\psi_{2}^{-1}\left(t\right)\leq C'\left(k\right)^{-1}\psi_{2}^{-1}\left(kt\right)$.
It follows that (\ref{eq:uniqueness result :  bounds}) is satisfied
with $C'_{1}=1$ and $C'_{2}=C'\left(\lambda^{-1}\right)^{-1}$. Then
$L'\left(s\right)$ is slowly varying and the result follows. The
proof for the super-homogeneous case follows along similar lines while
taking $\lambda\in(0,1]$. 
\end{proof}
Combining Lemma \ref{lem:Injection_result} and Lemma \ref{lem:uniqueness_result}
we obtain the following.
\begin{prop}
\label{prop:Injection_and_uniqueness}Let $\psi\in\mathfrak{B}$,
then the set of distributions $\mathfrak{L}_{s}^{\psi}$ is contained
in $\mathfrak{L}_{\psi}$. Moreover, if the L�vy measure of $\psi$
is sub-homogeneous or super-homogeneous then $\Phi_{\psi}^{-1}\left(\mathfrak{L}_{\psi}\right)=\mathfrak{L}_{s}$.
\end{prop}
We now apply \propref{Injection_and_uniqueness} to CTRW. 
\begin{prop}
\label{prop:time_changed_CTRW}Let $Y_{t}$ be a CTRW with i.i.d space-time
jumps $\left(J_{k},W_{k}^{\psi}\right)$ where $\left\{ W_{k}^{\psi}\right\} \in\mathfrak{L}_{s}^{\psi}$
and $\psi\left(s\right)\in\mathfrak{B}$. Then there exists a CTRW
$X_{t}$ with i.i.d space-time jumps $\left(J'_{k},W{}_{k}^{s}\right)$
where $\left\{ W_{k}^{s}\right\} \in\mathfrak{L}_{s}$ and an inverse
of a subrodinator with symbol $\psi\left(s\right)$ $E_{t}$ that
is independent of $\left\{ W_{k}^{s}\right\} $ s.t 
\begin{equation}
Y_{t}\overset{J_{1}}{\sim}\left(X_{E_{t}-}\right)^{+}.\label{eq:time changed CTRW}
\end{equation}
Conversely, assume $Y_{t}$ is a CTRW with waiting times $\left\{ W_{k}^{\psi}\right\} \in\mathfrak{L}_{\psi}$
s.t $Y_{t}\overset{J_{1}}{\sim}\left(X_{E_{t}-}\right)^{+}$, where
$X_{t}$ is a CTRW with waiting times $\left\{ W_{k}\right\} $ and
$E_{t}$ is the inverse of a subordinator of symbol $\psi\left(s\right)$
that is independent of $\left\{ W_{k}\right\} $. Moreover, assume
that $\psi\left(s\right)\in\mathfrak{B}$ has representation $\left(0,b,\mu\right)$
where $\mu$ is super-homogeneous or sub-homogeneous, then $W_{1}^{\psi}\in\mathfrak{L}_{s}^{\psi}$
and $W_{1}\in\mathfrak{L}_{s}$. 
\end{prop}
\begin{proof}
We note that if $T$ is a positive r.v then $\Phi_{\psi}\left(T\right)\sim D_{T}$
where $D_{t}$ is a subordinator of symbol $\psi$ independent of
$T$. Indeed, by the independence of $D_{t}$ and $T$ we have $\mathbb{E}\left(e^{-sD_{T}}\right)=\mathbb{E}\left(e^{-\psi\left(s\right)T}\right)=\hat{\Phi}_{\psi}\left(\mathcal{L}\left(T\right)\right)$.
Let $T_{n}^{\psi}=\sum_{k=1}^{n}W_{k}^{\psi}$ be the time of the
n'th jump of $Y_{t}$. Since $W_{1}^{\psi}\in\mathfrak{L}_{s}^{\psi}$
there exists a distribution $f^{s}\in\mathfrak{L}_{s}$ s.t $W_{1}^{\psi}\sim\Phi_{\psi}\left(f^{s}\right)$.
We now generate a sequence of i.i.d r.v's $\left\{ W_{k}'^{\psi}\right\} $
on a common probability space s.t $W_{1}'^{\psi}\sim W_{1}^{\psi}$.
Let $\left(\Omega,\mathcal{F},\mathbb{P}\right)$ be a probability
space, and let $\left\{ W_{i}^{s}\right\} $ be a sequence of i.i.d
random variables in $\Omega$ s.t $\Phi_{\psi}\left(W_{1}^{s}\right)\sim W_{1}^{\psi}$,
and let $T_{n}^{s}=\sum_{k=0}^{n}W_{k}^{s}$ . Let $D_{t}$ be a subordinator
of symbol $\psi\left(s\right)$ in $\left(\Omega,\mathcal{F},\mathbb{P}\right)$
independent of $\left\{ W_{i}^{s}\right\} $. Define $W_{k}'^{\psi}=D_{T_{k}^{s}}-D_{T_{k-1}^{s}}$,
and note that $\left\{ W_{k}'^{\psi}\right\} $ are i.i.d and $W_{1}'^{\psi}\sim W_{1}^{\psi}$.
Indeed, by the fact that $D_{t}$ is a strong Markov process, independent
of $T_{n}^{s}$ , with stationary increments we have
\begin{alignat*}{1}
W_{k}'^{\psi} & =D_{T_{k}^{s}}-D_{T_{k-1}^{s}}\\
 & \sim D_{T_{k}^{s}-T{}_{k-1}^{s}}\\
 & \sim D_{W_{k}^{s}}\sim\Phi_{\psi}\left(W_{k}^{s}\right)\\
 & \sim W_{k}^{\psi}.
\end{alignat*}
 By the independence of increments of $D_{t}$, we see that $W_{k}'^{\psi}$
are also independent. Assume now that $\left\{ J'_{k}\right\} $ are
i.i.d r.v's in $\Omega$ s.t $\left(J'_{i},W_{k}'^{\psi}\right)\sim\left(J_{i},W_{k}^{\psi}\right)$
and that $X_{t}$ is the CTRW associated with the space-time jumps
$\left(J_{i}',W_{i}^{s}\right)$. Let $T_{n}'^{\psi}=\sum_{k=1}^{n}W_{k}'^{\psi}$,
and define the process 
\[
Y_{t}'=\sum_{n=1}^{\infty}J_{n}'1\left(t\right)_{\left\{ y:T_{n}'^{\psi}\leq y\right\} }.
\]
Note that $Y'_{t}$ is a CTRW with space-time jumps $\left(J'_{i},W_{k}'^{\psi}\right)$
and therefore $Y'_{t}\overset{J_{1}}{\sim}Y_{t}$. Next we show (\ref{eq:time changed CTRW}).
Since $\psi$ is unbounded we see that $D_{t}$ is strictly increasing
and therefore $E_{t}$ is continuous, and it follows that a.s for
every $\omega\in\Omega$ we have 
\begin{alignat*}{1}
t\in\left\{ y:D_{T_{n}^{s}}\left(\omega\right)<y\right\}  & \iff t\in\left\{ y:T_{n}^{s}\left(\omega\right)<E_{y}\left(\omega\right)\right\} \\
 & \iff E_{t}\left(\omega\right)\in\left\{ y:T_{n}^{s}\left(\omega\right)<y\right\} ,
\end{alignat*}
and therefore
\begin{alignat}{1}
Y'_{t} & =\left(Y'_{t-}\right)^{+}\nonumber \\
 & =\left(\sum_{n=1}^{\infty}J_{i}'1\left(t\right)_{\left\{ y:D_{T_{n}^{s}}<y\right\} }\right)^{+}\nonumber \\
 & =\left(\sum_{n=1}^{\infty}J_{i}'1\left(E_{t}\right)_{\left\{ y:T_{n}^{s}<y\right\} }\right)^{+}\nonumber \\
 & =\left(X_{E_{t}-}\right)^{+}.\label{eq:time changed CTRW-1}
\end{alignat}
Now, suppose that $Y_{t}$ is a CTRW associated with the waiting times
$\left\{ W_{k}^{\psi}\right\} $, where $W_{1}^{\psi}\in\mathfrak{L}_{\psi}$
and $\psi$ has representation $\left(0,b,\mu\right)$ with $\mu$
being super-homogeneous or sub-homogeneous and that $Y_{t}\sim\left(X_{E_{t}-}\right)^{+}$,
where $X_{t}$ is a CTRW with space-time jumps $\left(J_{i},W_{i}\right)$
and $E_{t}$ is the inverse-subordinator of symbol $\psi$ independent
of $\left\{ W_{i}\right\} $. Let $T_{n}=\sum_{i=1}^{n}W{}_{i}$,
going backwards in equation (\ref{eq:time changed CTRW-1}), we see
that $W_{1}^{\psi}\sim\Phi_{\psi}W{}_{1}$. It is implied by \propref{Injection_and_uniqueness}
that $W_{1}\in\mathfrak{L}_{s}$ and the result follows. 
\end{proof}
\begin{rem}
In \cite{Meerschaert2011}, the mapping $\Phi_{\psi}$ was used implicitly
to obtain \emph{fractional Poisson processes}. Let $D_{t}$ be a subordinator
of symbol $\psi$, $E_{t}$ its inverse and let $N_{t}$ be a Poisson
process of intensity $1$. Then it was shown in \cite[Theorem 4.1]{Meerschaert2011}
that $N_{E_{t}}$ is a renewal process with waiting times $\left\{ W_{i}\right\} $
s.t 
\[
\mathbb{P}\left(W_{1}>t\right)=\mathbb{E}\left(e^{-\lambda E_{t}}\right).
\]
 
\end{rem}

\begin{rem}
\label{rem:strict inclusion}Let us say a distribution $f\left(dx\right)$
is a \emph{stable-mixture }if it is of the form 
\[
f\left(dx\right)=\integral 0{\infty}t^{-\nicefrac{1}{\alpha}}g\left(t^{-\frac{1}{\alpha}}x\right)p\left(dt\right)dx,
\]
 where $g\left(x\right)$ is the density of a standard stable r.v
of index $0<\alpha<1$ and $p$$\left(dt\right)$ is a measure whose
first moment (maybe infinite) is slowly varying . In other words,
$f\left(dx\right)$ is a stable-mixture if and only if $f\in\mathfrak{L}_{s}^{s^{\alpha}}$.
It is obvious that $\mathfrak{L}_{s}^{s^{\alpha}}\subsetneq\mathfrak{L}_{s^{\alpha}}$.
Firstly, distributions in $\mathfrak{L}_{s}^{s^{\alpha}}$ have densities
which may not be the case for distributions in $\mathfrak{L}_{s^{\alpha}}$.
Moreover, by (\ref{eq:characterization of the first layer}) we see
that whenever $\hat{f}\in\hat{\mathfrak{L}}_{\psi}$ s.t $\hat{f}\sim1-\psi\left(s\right)L\left(\frac{1}{s}\right)$
with $L$ a slowly varying function s.t $\lim_{s\rightarrow\infty}L\left(s\right)$
is zero or does not exist, $f\notin\mathfrak{L}_{s}^{\psi}$. Indeed,
\cite[Corollary 8.1.7 ]{bingham1989regular} states that if $L\left(t\right)$
is slowly varying then $\hat{f}\left(s\right)\sim1-sL\left(s^{-1}\right)$
is equivalent to $\int_{0}^{t}ydf\left(y\right)\sim L\left(t\right)$
hence $L$ must be increasing. A natural question is whether $\mathfrak{L}_{s}^{s^{\alpha}}$
is weakly dense in $\mathfrak{L}_{s^{\alpha}}$? Unfortunately we
could not answer that. We could not even answer what appears to be
a simpler version of that question, namely, if $0<a<b$ and $A_{a}^{b}=\left\{ \hat{f}\in\hat{\mathfrak{L}}:\hat{f}\left(s\right)\sim1-cs,a\leq c\leq b\right\} $
, $B_{a}^{b}=\left\{ \hat{f}\in\hat{\mathfrak{L}}:\hat{f}\left(s\right)\sim1-c\psi\left(s\right),a\leq c\leq b\right\} $
is $\Phi_{\psi\left(s\right)}\left(A_{a}^{b}\right)$ weakly dense
in $B_{a}^{b}$? 

\begin{rem}
\label{rem:independent outer and inner processes}In the case where
$A_{t}$ and $E_{t}$ are independent, by the fact that L�vy process
are stochastically continuous we see that the $A_{E_{t}}$ and $\left(A_{E_{t}-}\right)^{+}$
have the same law. 
\end{rem}
\end{rem}
Remark \ref{rem:strict inclusion} underlines the possibly limited
range of measures in $\mathfrak{L}_{s}^{\psi}$ compared to $\mathfrak{L}_{\psi}$.
In order to extend the set $\mathfrak{L}_{s}^{\psi}$ we may use $\Phi_{\psi'}$
where $\psi'\left(s\right)\in\mathfrak{B}$ s.t $\psi'\left(s\right)\sim\psi L'\left(s^{-1}\right)$
where $L'\left(t\right)$ is slowly varying. As the product of two
slowly varying functions is a slowly varying function we must have
$\Phi_{\psi'}\left(\mathfrak{L}_{s}\right)\subset\mathfrak{L}_{\psi}$.
Indeed, if $\hat{f}\left(s\right)\sim1-sL\left(s^{-1}\right)$ where
$L\left(s^{-1}\right)$ is slowly varying then $\hat{f}\left(\psi'\left(s\right)\right)\sim1-\psi\left(s\right)L'\left(s^{-1}\right)L\left(\psi'\left(s\right)^{-1}\right)$
and $\hat{f}\left(\psi\left(s\right)\right)\in\mathfrak{L}_{\psi}$.
Define the set 
\[
\mathfrak{B}_{\psi}:=\left\{ \psi'\left(s\right)\in\mathfrak{B}:\psi'\left(s\right)\sim\psi\left(s\right)L\left(s^{-1}\right),L\text{ is slowly varying}\right\} ,
\]
and then define
\[
\mathfrak{L}_{s}^{\bar{\psi}}:=\cup_{\psi'\in\mathfrak{B}_{\psi}}\Phi_{\psi'}\left(\mathfrak{L}_{s}\right).
\]
\\
Note that the mapping $\Phi_{s^{\alpha}}$ reduces the ``regularity''
$s$ around $s=0$ for $\hat{f}\left(s\right)\in\mathfrak{\hat{L}}_{s}$
with the more coarse ``regularity'' $s^{\alpha}$. In order to maintain
general results we make the following assumption on $\psi\left(s\right)$. 
\begin{assumption}
\label{assu:regularity_0}We assume $\psi\left(s\right)$ satisfies
\begin{equation}
\lim_{s\rightarrow0^{+}}\frac{s}{\psi\left(s\right)}=0.\label{eq:less regularity}
\end{equation}
\end{assumption}
Note that due to the relation between the regularity of the LT $\hat{f}$
around zero and the moments of the distribution $f$ we see that if
$f\in\mathfrak{L}_{\psi}$ where $\psi$ satisfies (\ref{eq:less regularity})
then the first moment of $f$ is infinite. It turns out that the set
of distribution $\mathfrak{L}_{s}^{\bar{\psi}}$ is indeed rich in
$\mathfrak{L}_{\psi}$.
\begin{lem}
\label{lem:Pseudo_smooth_dense}Let $\psi\in\mathfrak{B}$ that satisfies
(\ref{eq:less regularity}).Then the set of distributions $\mathfrak{L}_{s}^{\bar{\psi}}$
is weakly dense in $\mathfrak{L}_{\psi}$.
\end{lem}
\begin{proof}
Let $Y\in\mathfrak{L}_{\psi},$ that is, $\mathbb{E}\left(e^{-sY}\right)\sim1-\psi\left(s\right)L\left(s^{-1}\right)$.
Define $Y_{n}=Y1_{[0,n]}$ and note that $Y_{n}\in\mathfrak{L}_{s}$.
Next define 
\begin{equation}
\psi_{n}\left(s\right)=s+\mu_{n}^{-1}\int_{0}^{\infty}\left(1-e^{-sy}\right)f\left(dy\right),\label{eq:Pseudo-smooth dense-2}
\end{equation}
where $\mu_{n}=\mathbb{E}\left(Y_{n}\right)$ and $f\left(dy\right)$
is the distribution of $Y$. Since $\psi$ satisfies (\ref{eq:less regularity})
we see that $\mu_{n}\rightarrow\infty$ by monotone convergence. It
follows that 
\begin{equation}
\psi_{n}\left(s\right)\rightarrow s,\label{eq:Pseudo-smooth dense}
\end{equation}
for every $s>0$. Moreover, denote by $\bar{f}\left(t\right)=\int_{t}^{\infty}f\left(dy\right)$
the tail of the distribution $f\left(dy\right)$. It is straightforward
to verify that $\hat{\bar{f}}\left(s\right):=\int e^{-sy}\bar{f}\left(y\right)dy=\frac{1-\hat{f}\left(s\right)}{s}$
and therefore that $\hat{\bar{f}}\left(s\right)\sim s^{-1}\psi\left(s\right)L\left(s^{-1}\right)$.
Using integration by parts in (\ref{eq:Pseudo-smooth dense-2}) we
see that for every $n$
\[
\psi_{n}\left(s\right)\sim s+\mu_{n}^{-1}\psi\left(s\right)L\left(s^{-1}\right).
\]
Let $f_{n}$ be the distribution of $Y_{n}$. Since $\hat{f}_{n}\left(s\right)\sim1-\mu_{n}s$,
it follows by (\ref{eq:less regularity}) that
\begin{equation}
\hat{f}_{n}\left(\psi_{n}\left(s\right)\right)\sim1-\psi\left(s\right)L\left(s^{-1}\right),\label{eq:Pseudo-smooth dense-1}
\end{equation}
 in particular, $\Phi_{\psi_{n}}f_{n}\in\mathfrak{L}_{\psi}$. It
is left to show that $\Phi_{\psi_{n}}f_{n}\rightarrow f$ as $n\rightarrow\infty$.
But this follows easily form (\ref{eq:Pseudo-smooth dense}) and the
fact that $Y_{n}$ converges weakly to $Y$. 
\end{proof}
\begin{rem}
The reason why we did not use $\psi_{n}=s+\mu_{n}^{-1}\psi\left(s\right)$
instead of the form in (\ref{eq:Pseudo-smooth dense-2}) is that the
form in (\ref{eq:Pseudo-smooth dense-2}) has an advantage when $\psi\left(s\right)=s^{\alpha}L\left(s^{-1}\right)$
where $L$$\left(t\right)$ is slowly varying. Indeed, by Karamata's
Theorem we know that $\mathbb{E}\left(e^{-sY}\right)\sim1-s^{\alpha}L\left(s^{-1}\right)\Gamma\left(1-\alpha\right)$
is equivalent to $\mathbb{P}\left(Y>t\right)\sim t^{-\alpha}L\left(t\right)$
. It follows that for every $n$ our approximation $\Phi_{\psi_{n}}Y_{n}$
of $Y$ satisfies 
\begin{equation}
\lim_{t\rightarrow\infty}\frac{\mathbb{P}\left(Y>t\right)}{\mathbb{P}\left(\Phi_{\psi_{n}}Y_{n}>t\right)}=1.\label{eq:same tail}
\end{equation}
Equation \ref{eq:same tail} will be utilized in the sequel in order
to obtain quantitative results. 
\end{rem}
\begin{prop}
\label{prop:almost_every_CTRW_is_time_changed_CTRW}Let $Y_{t}^{n}$
be the CTRW associated with the i.i.d space-time jumps $\left(J_{i}^{n},a_{n}W_{i}\right)$
where $W_{1}\in\mathfrak{L}_{s}^{\overline{s^{\alpha}}}$ . Then there
exists a CTRW $X_{t}^{n}$ associated with the i.i.d space-time jumps
$\left(J_{i}^{n},n^{-1}U_{i}\right)$ s.t $U_{1}$ has finite mean,
and a sequence of inverse-subordinators $E_{t}^{n}$ independent of
$\left\{ U_{i}\right\} $ so that for every $n$
\[
Y_{t}^{n}\overset{J_{1}}{\sim}\left(X_{E_{t}^{n}-}^{n}\right)^{+},
\]
and s.t $E_{t}^{n}$ converges in law w.r.t the $J_{1}$-topology
to $E_{t}$, the inverse of a stable subordinator of index $\alpha$. 
\end{prop}
\begin{proof}
Since $W_{1}\in\mathfrak{L}_{s}^{\overline{s^{\alpha}}}$, there exists
$\psi\in\mathfrak{B}_{s^{\alpha}}$(we shall use the one in \lemref{Pseudo_smooth_dense})
s.t $W_{1}\in\mathfrak{L}_{s}^{\psi}$. By \propref{time_changed_CTRW}
we see that 
\[
Y_{t}^{1}=\left(X_{E_{t}^{1}-}^{1}\right)^{+},
\]
where $E_{t}^{1}$ is the inverse-subordinator of symbol $\psi_{1}=sa_{n}b+\int_{0}^{\infty}\left(1-e^{-sa_{n}y}\right)\mu\left(dy\right)$
and $X_{t}^{1}$ is a CTRW associated with i.i.d space-time jumps
$\left(J_{i}^{1},U_{i}\right)$ with $U_{1}\in\mathfrak{L}_{s}$.
By Proposition \ref{prop:time_changed_CTRW} it is enough to show
that 
\begin{equation}
a_{n}W_{1}\sim\Phi_{\psi_{n}}\left(n^{-1}U_{1}\right),\label{eq:main thm-2}
\end{equation}
where $\psi_{n}$ is the symbol of a strictly increasing subordinator.
Looking at the Laplace Transform of $a_{n}W_{1}$ we see that 
\begin{align}
\mathbb{E}\left(e^{-sa_{n}W_{1}}\right) & =\mathbb{E}\left[e^{-U_{1}\left(sa_{n}b+\int_{0}^{\infty}\left(1-e^{-sa_{n}y}\right)\mu\left(dy\right)\right)}\right]\nonumber \\
 & =\mathbb{E}\left[e^{-n^{-1}U_{1}\left(sa_{n}nb+\int_{0}^{\infty}\left(1-e^{-sy}\right)n\mu\left(a_{n}^{-1}dy\right)\right)}\right],\label{eq:pseudo_self_similarity}
\end{align}
which implies (\ref{eq:main thm-2}) with $\psi_{n}\left(s\right)=sna_{n}b+\int_{0}^{\infty}\left(1-e^{-sy}\right)n\mu\left(a_{n}^{-1}dy\right)$.
Letting $E_{t}^{n}$ be the inverse of a strictly increasing subordinator
of symbol $\psi_{n}$ and invoking again \propref{time_changed_CTRW}
we see that 
\[
Y_{t}^{n}=\left(X_{E_{t}^{n}-}^{n}\right)^{+}.
\]
We are left to show that $E_{t}^{n}$ converges in law to $E_{t}$,
the inverse of a stable subordinator of index $\alpha$. To see that,
first note that by the definition of $\psi_{1}$ and Karamata's Theorem
we know that $\bar{\mu}_{1}\left(y\right)\sim L\left(y\right)y^{-\alpha}$.
Let $h\left(y\right)$ be a smooth function with compact support $[a,b]\subset\mathbb{R}_{+}/\left\{ 0\right\} $,
then by (\ref{eq:the sequence a_n})
\begin{align*}
\lim_{n\rightarrow\infty}\mathbb{\int}_{0}^{\infty}h\left(y\right)n\mu_{1}\left(a_{n}^{-1}dy\right) & =\lim_{n\rightarrow\infty}\mathbb{\int}_{a}^{b}\frac{\partial h\left(y\right)}{\partial y}n\bar{\mu}_{1}\left(a_{n}^{-1}y\right)dy\\
 & =\mathbb{\int}_{a}^{b}\frac{\partial h\left(y\right)}{\partial y}\frac{y^{-\alpha}}{\Gamma\left(1-\alpha\right)}dy,
\end{align*}
and form the fact that $a_{n}n\rightarrow0$ ($a_{n}$ is regularly
varying with parameter $-\frac{1}{\alpha}$) we see that $\mu_{n}$
converges vaguely to $\mu\left(dy\right)=\frac{\alpha}{\Gamma\left(1-\alpha\right)}y^{-\alpha-1}dy$,
the L�vy measure of a standard stable subordinator of index $\alpha$.
However, convergence of characteristics of Feller processes implies
weak convergence of their law in the $J_{1}$ topology. In other words,
if $D_{t}^{n}$ is the subordinator whose symbol is $\psi_{n}$ we
see that $D_{t}^{n}\overset{J_{1}}{\Rightarrow}D_{t}$ . Next we use
\lemref{Straka_and_Henry} with $\left(n^{-1},D_{t}^{n}\right)$ to
obtain 
\[
E^{n}\overset{J_{1}}{\Rightarrow}E,
\]
This completes the proof. 
\end{proof}
\propref{almost_every_CTRW_is_time_changed_CTRW} can be understood
in the following way; let $A_{t}$ be an increasing process and let
$A\left(\omega,T\right)$ be the regenerative set of $A_{t}$ in the
interval $[0,T]$ (we may also consider $[0,\infty)$). That is,
\[
A\left(\omega,T\right)=\left\{ u\in[0,T]:A_{u-\epsilon}\left(\omega\right)<A_{u}\left(\omega\right)<A_{u+\epsilon}\left(\omega\right),\forall\epsilon>0\right\} .
\]
 Note that the mapping $\text{\ensuremath{\Phi}}_{\psi}$ can be viewed
as a mapping on processes. Let $X_{t}$ be a process, then we define
\[
\text{\ensuremath{\Phi}}_{\psi}\left(X_{t}\right)=X_{E_{t}},
\]
where $E_{t}$ is the inverse-subordinator of symbol $\psi$ independent
of $X_{t}$. Moreover, $\Phi_{\psi}$ can also be viewed as a mapping
between regenerative set-valued random variables. That is, conditioned
on $E\left(\omega,\infty\right)$, $\text{\ensuremath{\Phi}}_{\psi}\left(X_{t}\right)\left(\cdot,\infty\right)$
is a random regenerative set contained in $E\left(\omega,\infty\right)$.
Lastly, note that conditioned on $E_{t}\left(\omega\right)=\xi$,
$\Phi_{\psi}$ can be viewed as a function $\Phi_{\psi,\xi}:\mathfrak{L}\rightarrow\mathfrak{L}$.
If $U\in\mathfrak{L}$ has distribution $\mu$ 
\[
\Phi_{\psi,\xi}\left(\mu\right)\sim\xi_{U}^{-1}.
\]
$\Phi_{\psi,\xi}$ sends measures in $\mathfrak{L}$ to measures whose
support is in the regenerative set of $\xi$. Let $\mu\in\mathfrak{L}$,
and let $f_{t}$ be a time-change, then we define the probability
measure $\mu_{f}$ on Borel sets of $\mathbb{R}_{+}$ to be 
\begin{equation}
\mu_{f}\left(A\right)=\mu\left(A_{f^{-1}}\right),\label{eq:mu_f}
\end{equation}
for every Borel set in $A\subset\mathbb{R}_{+}$, where for an increasing
$f$ $A_{f}$ is the set 
\[
A_{f}=\left\{ x\in\mathbb{R}:f\left(x\right)\in A\right\} .
\]
We have $\Phi_{\psi,\xi}\left(\mu\right)=\mu_{\xi}\left(dx\right)$.
If $U_{1}$ has finite mean then by the SLLN of Renewal Theory we
know that with probability one the regenerative points of the CTRW
$T_{t}^{n}$ associated with the space-time jumps $\left(1,n^{-1}U_{i}\right)$
'converge' to a set that is dense in $[0,T]$, namely 
\[
T\left(\omega,T\right)=\cup_{n}T^{n}\left(\omega,T\right).
\]
Since $D_{t}\left(\omega\right)$ is right continuous we deduce that
the mapping $\Phi_{\psi}$ is 'continuous' (if $x_{n}\in T\left(\omega,T\right)$
s.t $x_{n}>x$ and $x_{n}\rightarrow x$ then $D_{x_{n}}\rightarrow D_{x}$)
and $E\left(\omega,T\right)$ is a perfect set (closed, with no isolated
points). It follows that $\Phi_{\psi}\left(T\left(\omega,T\right)\right)$
is dense in $E\left(\omega,T\right)$. In other words, as $n\rightarrow\infty$
the trajectory $E_{t}\left(\omega\right)$ is delineated by the regenerative
points of $T_{t}$. This idea holds more generally. Let $f\in\mathfrak{L}$
and let $\left\{ U_{i}\right\} $ be i.i.d r.vs with distribution
$f$. Let us define $T_{0}=0$ and
\[
T_{n}=\sum_{i=1}^{n}U_{i}.
\]
We say that $f$ is \emph{relatively stable }(\cite[8.8]{bingham1989regular})if
there exist norming constants $a_{n}$ s.t 
\[
a_{n}T_{n}\rightarrow1,
\]
where convergence is in probability. Next define the renewal process
\[
N_{t}=\max\left\{ k:T_{k}\leq t\right\} .
\]
Define the \emph{residual lifetime} $Z_{t}$ and the \emph{aging}
$Y_{t}$ by 
\begin{align*}
Y_{t} & =t-T_{N_{t}}\\
Z_{t} & =T_{N_{t}+1}-t.
\end{align*}
Finally, we let $a_{n}>0$ be any sequence s.t 
\[
1-\hat{f}\left(a_{n}\right)\sim n^{-1}.
\]
 The following is known \cite[Theorem 8.8.1]{bingham1989regular}. 
\begin{lem}
\label{lem:relative_stability}Let $f\in\mathfrak{L}$ and let $Y_{t}$
and $Z_{t}$ be the aging and the residual lifetime processes associated
with $f$. The following are equivalent:
\begin{enumerate}
\item $f\in\mathfrak{L}_{s}$.
\item $f$ is relatively stable.
\item $\frac{Y_{t}}{t}\rightarrow0$ in probability.
\item $\frac{Z_{t}}{t}\rightarrow0$ in probability.
\end{enumerate}
\end{lem}
Define $i_{f}^{n}=\sup\left\{ i:T_{i}^{n}\leq T\right\} $ and the
set 
\begin{equation}
A_{\delta,T}^{n}=\left\{ \omega:\sup_{1\leq i\leq i_{f}^{n}}\left|T_{i}^{n}-T_{i-1}^{n}\right|<\delta\right\} .\label{eq:A_delta}
\end{equation}
The set $A_{\delta,T}^{n}$ is the event that one can not find two
consecutive regenerative points whose distance is larger than $\delta$.
We shall need the next lemma. 
\begin{lem}
\label{lem:A_delta}Let $f\in\mathfrak{L}_{s}$ and $T_{i}^{n}=a_{n}\sum_{j=1}^{i-1}U_{j}$
where $\left\{ U_{i}\right\} $ are i.i.d and $U_{1}\sim f$. Then
for every $\delta>0$ $\mathbb{P}\left(A_{\delta}^{n}\right)\rightarrow1$
as $n\rightarrow\infty$. 
\end{lem}
\begin{proof}
Assume w.l.o.g that $T=1$. Define the set $\left\{ t_{k}\right\} _{k=0}^{2^{m}-1}$
where $t_{k}=2^{-m}k$ where $m=\left\lceil \log_{2}\delta^{-1}\right\rceil +1$.
If $N_{t}^{n}$ is the renewal process of $\left\{ T_{i}^{n}\right\} $
and $Z_{t}^{n}=T_{N_{t}^{n}+1}-t$ is its residual lifetime, then
for every $\epsilon>0$, by \lemref{relative_stability}, we have
for large enough $n,$
\[
\mathbb{P}\left(Z_{t_{k}}^{n}>2^{-m}\right)<2^{-m}\epsilon\qquad\forall t_{k},1\leq k<2^{m}.
\]
It is left to note that $\left\{ \cup_{1\leq k<2^{m}}\left\{ Z_{t_{k}}>\nicefrac{\delta}{2}\right\} \right\} ^{c}\subset A_{\delta,T}^{n}$.
\end{proof}
\begin{prop}
\label{prop:general_scheme}Assume 
\begin{equation}
\left(A_{t}^{n},D_{t}^{n}\right)\overset{J_{1}}{\Rightarrow}\left(A_{t},D_{t},\right)\label{eq:general scheme-4}
\end{equation}
where $D_{t}$ is a.s strictly increasing. Let $\left\{ U_{i}\right\} $
be i.i.d r.vs independent of the sequence $\left(A_{t}^{n},D_{t}^{n}\right)$
where $U_{1}\in\mathfrak{L}_{s}$. Let $T_{i}^{n}=a_{n}\sum_{j=1}^{i}U_{j}$
be the renewal epoch and let $\left\{ X_{t}^{n}\right\} _{n=1}^{\infty}$
be the CTRWs associated with the space-time jumps 
\[
\left\{ J_{i}^{n},W_{i}^{n}\right\} =\left\{ A_{T_{i}^{n}}^{n}-A_{T_{i-1}^{n}}^{n},D_{T_{i}^{n}}^{n}-D_{T_{i-1}^{n}}^{n}\right\} .
\]
Then
\begin{equation}
X_{t}^{n}\overset{J_{1}[0,T]}{\Rightarrow}\left(A_{E_{t}-}\right)^{+}\label{eq:general scheme-2}
\end{equation}
where $E_{t}$ is the generalized inverse of $D_{t}$. 
\end{prop}
\begin{proof}
Let $\epsilon>0$. If $E^{n}$ are the generalized inverses of $D^{n}$,
Due to (\ref{eq:general scheme-4}) we see that $E_{T}^{n}\Rightarrow E_{T}$
and therefore, one can find $\tilde{T}>0$ s.t 
\[
\sup_{n}\mathbb{P}\left(E_{T}^{n}>\tilde{T}\right)<\frac{\epsilon}{3}.
\]
 For every $\delta>0$, define the event $A_{\delta,\tilde{T}}^{n}$
as in (\ref{eq:A_delta}). Consider the CTRW $Y_{t}^{n}$ associated
with the time-space jumps $\left(\left(J_{i}^{n},W_{i}^{n}\right),a_{n}U_{i}\right)$
(note that $Y_{t}^{n}\in\mathbb{R}^{d}\times\mathbb{R}_{+})$. We
now claim that for large enough $n$ , we have 
\begin{equation}
\rho_{d_{J_{1}[0,\tilde{T}]}}\left(\left(A_{t}^{n},D_{t}^{n}\right),Y_{t}^{n}\right)<\epsilon.\label{eq:general scheme-1}
\end{equation}
 Recall that if $f\in\mathbb{D}[0,\tilde{T}]$, then the modulus of
continuity of $f$ is given by 
\begin{align*}
\omega_{f}^{\tilde{T}}\left(\delta\right) & =\inf\left\{ \max_{1\leq i\leq m}\theta_{f}[t_{i-1},t_{i}):\exists m\geq1,\right.\\
 & \left.0=t_{0}<t_{1}...<t_{m}=\tilde{T}\text{ s.t. }t_{i}-t_{i-1}>\delta\text{ for all }i\leq m\right\} ,
\end{align*}
where 
\[
\theta_{f}[s,t)=\sup_{s\leq u<w\leq t}\left|f\left(u\right)-f\left(w\right)\right|.
\]
Define 
\[
B_{\delta,\tilde{T}}^{n}=\left\{ \omega_{\left(A_{t}^{n},D_{t}^{n}\right)\left(\omega\right)}^{\tilde{T}}\left(\delta\right)<\epsilon\right\} ,
\]
 assumption (\ref{eq:general scheme-4}), suggests that for every
$\epsilon>0$ there exists $\delta>0$ s.t 
\begin{equation}
\sup_{n}\mathbb{P}\left(B_{\delta,\tilde{T}}^{n}\right)>1-\frac{\epsilon}{3}.\label{eq:general scheme}
\end{equation}
Define the sequence $f^{n}\in\mathbb{D}_{\mathbb{R}^{d}\times\mathbb{R}^{+}}[0,\tilde{T}]$
by $f_{t}^{n}=\left(A_{T_{i}}^{n},D_{T_{i}}^{n}\right)$ on $T_{i}\leq t<T_{i+1}$
and let $\delta'<\frac{\delta}{2\tilde{T}}\min\left(\delta,\epsilon\right)$.
We first condition on $A_{\delta',\tilde{T}}^{n}\cap B_{\delta,\tilde{T}}^{n}$
, $\left(A_{\cdot}^{n},D_{\cdot}^{n}\right)$, and $\left\{ E_{T}^{n}>\tilde{T}\right\} $
i.e. we would like to show that
\begin{equation}
\mathbb{P}\left(d_{J_{1}}\left(\left(A_{\cdot}^{n},D_{\cdot}^{n}\right),f_{\cdot}^{n}\right)>\epsilon|A_{\delta',\tilde{T}}^{n}\cap B_{\delta,\tilde{T}}^{n},\left(A_{\cdot}^{n},D_{\cdot}^{n}\right),\left\{ E_{T}^{n}>\tilde{T}\right\} \right)=0.\label{eq:general scheme-3}
\end{equation}
Indeed, by (\ref{eq:general scheme}), on $B_{\delta,\tilde{T}}^{n}$
one can find $0=t_{0}<t_{1}...<t_{m}=\tilde{T}$ s.t for every $n\geq1$,
$t_{i}-t_{i-1}>\delta$ and $\theta_{\left(A_{\cdot}^{n},D_{\cdot}^{n}\right)}[t_{i-1},t_{i})<\epsilon$
for $1\leq i\leq m$. Let $T^{n}=\left\{ T_{i}^{n}:T_{i}^{n}\leq\tilde{T}\right\} $.
On $A_{\delta',\tilde{T}}^{n}$ one can find the two points 
\begin{align*}
T_{i}^{n,1} & =\inf\left\{ T^{n}\cap[t_{i},t_{i+1})\right\} \\
T_{i}^{n,2} & =\sup\left\{ T^{n}\cap[t_{i},t_{i+1})\right\} ,
\end{align*}
 s.t $t_{i}\leq T_{i}^{n,1}<T_{i}^{n,2}<t_{i+1}$ for $0\leq i\leq m-1$.
The distance between $T_{i}^{n,2}$ and $T_{i+1}^{n,1}$ is at most
$\delta'$ and so one can find a homeomorphism $\lambda:[0,\tilde{T}]\rightarrow[0,\tilde{T}]$
s.t $\lambda\left(T_{i}^{n,1}\right)=t_{i}$ and s.t $\sup\left|\lambda\left(s\right)-s\right|\leq m\delta'\leq\frac{\tilde{T}}{\delta}\delta'<\epsilon$
(one simply maps the interval $[T_{i}^{n,2},T_{i+1}^{n,1})$ to $[T_{i}^{n,2},t_{i+1})$
which costs no more then $\text{\ensuremath{\delta}}'$ as $\left|T_{i}^{n,2}-T_{i+1}^{n,1}\right|<\delta'$)
. Next note that by the definition of $\omega_{f}^{\tilde{T}}\left(\delta\right)$
and (\ref{eq:general scheme}) we see that on $A_{\delta',\tilde{T}}^{n}\cap B_{\delta,\tilde{T}}^{n}$
\begin{align*}
\sup_{0\leq s\leq\tilde{T}}\left|f_{\lambda\left(s\right)}^{n}-\left(A_{s}^{n},D_{s}^{n}\right)\right| & <\epsilon,\\
\sup_{0\leq s\leq\tilde{T}}\left|\lambda\left(s\right)-s\right| & <\epsilon.
\end{align*}
Hence (\ref{eq:general scheme-3}) holds. By \lemref{A_delta}, for
large enough $n$ 
\[
\mathbb{P}\left(A_{\delta}^{n}\right)>1-\frac{\epsilon}{3}.
\]
Taking expectation in (\ref{eq:general scheme-3}) while using independence
we conclude that for large enough $n$
\[
\mathbb{P}\left(d_{J_{1}[0,\tilde{T}]}\left(\left(A_{\cdot}^{n},D_{\cdot}^{n}\right),f_{\cdot}^{n}\right)>\epsilon\right)<\epsilon,
\]
or (\ref{eq:general scheme-1}) which implies that 
\[
f_{t}^{n}\overset{J_{1}[0,\tilde{T}]}{\Rightarrow}\left(A_{t},D_{t}\right).
\]
From here we use \lemref{Straka_and_Henry} to obtain $X_{t}^{n}\overset{J_{1}[0,E_{T}]}{\Rightarrow}\left(A_{E_{t}-}\right)^{+}$
.
\end{proof}
\begin{rem}
If $U_{1}$ in \propref{general_scheme} has finite mean and $\left(A_{t}^{n},D_{t}^{n}\right)\overset{J_{1}}{\rightarrow}\left(A_{t},D_{t},\right)$
a.s, then using the SLLN of Renewal Theory and same arguments as in
\propref{general_scheme} we see that conditioned on $\left\{ A_{t}^{n},D_{t}^{n}\right\} _{n=1}^{\infty}$
we have $X_{t}^{n}\overset{J_{1}}{\rightarrow}\left(A_{E_{t}-}\right)^{+}$
with probability 1. 
\end{rem}
As we have shown that CTRW with heavy tailed waiting times can be
represented as CTRW with finite mean waiting times subordinated to
a time-change, we see that CTRWs � la Montroll and Weiss are essentially
CTRWs in random environment. Among the well known Random Walks in
Random Environment(RWRE) are the so-called trap models. The most basic
of which is arguably the Bouchaud model. The most basic setup consists
of a simple graph $G=\left(V,E\right)$ where $V$ is the set of vertices
and $E$ is the set of edges. We are also given the trapping environment
\[
\boldsymbol{\tau}=\left\{ \tau_{x}>0:x\in V\right\} .
\]
On the graph $G$ we preform a CTRW with exponential waiting times
whose jump rate is given by 
\[
w_{xy}=\left\{ \begin{array}{cc}
\tau_{x}^{-1}, & \left(x,y\right)\in E,\\
0, & \text{otherwise}.
\end{array}\right.,
\]
and the generator is given by 
\begin{equation}
Lf\left(x\right)=\sum_{y\sim x}w_{x,y}\left(f\left(y\right)-f\left(x\right)\right).\label{eq:generator}
\end{equation}
In words, the larger $\tau_{x}$ is, the deeper the trap at site $x$
and the longer the CTRW stays at the site $x$. In order to obtain
a non-trivial (simple random walk on $G$) limit we assume that $\left\{ \tau_{x}\right\} $
are i.i.d and that $\tau_{x}\in\mathfrak{L}_{s^{\alpha}}$. In \cite{fontes2002random}
Fontes et al studied the Bouchaud model where $G$ is $\mathbb{Z}$
with nearest neighbor edges. The Markov process $X_{t}$ associated
with the generator (\ref{eq:generator}) (conditioned on the environment
$\boldsymbol{\tau}$) is called the \emph{quenched} process. Taking
expectation w.r.t the law of $\boldsymbol{\tau}$ we obtain the \emph{annealed}
process. Let $a_{n}$ be the sequence defined in (\ref{eq:the sequence a_n}).
One is interested in the limit (in distribution) of the Bouchaud model
\begin{equation}
n^{-1}X_{tna_{n}^{-1}}\rightarrow X_{t}.\label{eq:Bouchaud limit}
\end{equation}
It was proven in \cite{fontes2002random} that $X_{t}$ is a Brownian
motion time-changed by the generalized inverse of the local time of
a standard Brownian motion integrated against a Poisson measure on
$\mathbb{R}\times\mathbb{R}_{+}$ with intensity $\alpha t^{-\alpha-1}1\left(t\right)_{(0,\infty)}dtdx$.
This was referred to as \emph{Singular Diffusion}. It turns out that
the dimension of the lattice affects the limit in (\ref{eq:Bouchaud limit})(although
the scaling is different). Indeed, it was proven in \cite{Arous2007}
that under proper scaling of the Bouchaud model on $\mathbb{Z}^{d}$
for $d>1$ the limit is $B_{E_{t}}$ , i.e. a Brownian motion time-changed
by the inverse of a standard stable subordinator independent of $B_{t}$
(this is referred to as \emph{Fractional Kinetics}) . It is worth
mentioning here that the scaling in dimension $d>2$  is the same
as that of the CTRW in the sense of Montroll and Wiess. The limit
of the Bouchaud model for dimension larger than one is the same as
that in the uncoupled Montroll and Wiess CTRW model. \propref{almost_every_CTRW_is_time_changed_CTRW}
suggests that the CTRW in the sense of Montroll and Wiess with waiting
times in $\mathfrak{L}_{s^{\alpha}}$ has a representations as annealed
process of possibly two different RWRE. Consider a probability space
$\left(\Omega,\mathcal{F},\mathbb{P}\right)$ on which there exists
a random continuous time-change (continuous increasing processes)
$E_{t}$. We also have a CTRW $\mathring{X}_{t}\in\mathbb{Z}^{d}$
associated with the i.i.d space-time jumps $\left(J_{i},U_{i}\right)$
where $U_{1}$ has finite mean, $\left\{ U_{i}\right\} $ is independent
of $E_{t}$, and where the probability transition function $p_{t}\left(\left(J_{i},U_{i}\right)\in\left(dx,du\right)\right)$
may depend on time and the random enviornment. Unless $J_{i}$ and
$U_{i}$ are independent $\mathring{X}_{t}$ need not be Markovian
even if $U_{1}\sim Exp\left(\lambda\right).$ Given a realization
of the time change $E_{t}\left(\omega\right)$ (our random environment)
we consider the process (the quenched process) 
\begin{equation}
X_{t}^{\lyxmathsym{\mbox{I}}}=\mathring{X}_{E_{t}\left(\omega\right)}.\label{eq:RWRE_Type_1}
\end{equation}
 We refer to $X_{t}^{\text{\mbox{I}}}$ in (\ref{eq:RWRE_Type_1})
as the quenched process of RWRE of type \mbox{I}. We will say that
$\tilde{X}_{t}^{\text{\mbox{I}}}$ is an annealed process of RWRE
of type \mbox{I} if there exists a CTRW $\mathring{X}_{t}$ s.t 
\[
\mathbb{P}\left(\tilde{X}_{\cdot}^{\lyxmathsym{\mbox{I}}}\in dx_{\cdot}\right)=\int\mathbb{P}\left(\mathring{X}_{\xi_{\cdot}}\in dx_{\cdot}\right)P_{E}\left(d\xi_{\cdot}\right),
\]
where $P_{E}\left(d\xi_{\cdot}\right)$ is the law of our random environment
$E_{t}$, that is 
\[
\int_{A}P_{E}\left(d\xi_{\cdot}\right)=\mathbb{P}\left(E\in A\right),
\]
with $A$ a Borel set in the Borel sigma-algebra of $\mathbb{D}$.
We are interested in the limit
\begin{equation}
n^{-1}\tilde{X}_{tn^{\frac{2}{\alpha}}}^{\text{\mbox{I}}}\Rightarrow\bar{X}_{t}^{\lyxmathsym{\mbox{I}}}.\label{eq:scaling_CTRW-1}
\end{equation}
Next we introduce another RWRE model which is somewhat of a temporal
trap model. Let 
\begin{equation}
\boldsymbol{\tau}=\left\{ \tau_{n}>0:n\in\mathbb{Z}_{+}\right\} ,\label{eq:random_enviornment}
\end{equation}
be our random temporal landscape. We also assume the existence of
a family of probability transition functions 
\begin{equation}
p_{t}\left(s,x;y\right)\qquad t>0.\label{eq:transition_probability}
\end{equation}
Let $X_{t}$ be the CTRW who after the n'th jump ends up at site $x$
and waits an exponential time $s$ of mean $\tau_{n}$, and then makes
a spatial jump to one of its neighbors according to a distribution
$p_{\tau_{n}}\left(s;y\right)$. In other words, the temporal landscape
$\boldsymbol{\tau}$ affects both the temporal dynamics as well as
the spatial. More precisely, assume we have a sequence of positive
r.v $\left\{ \tau_{i}\right\} $ and let $\left\{ U_{i}\right\} $
be a sequence of i.i.d waiting times s.t $\mathbb{E}\left(U_{1}\right)=1$
independent of $\left\{ \tau_{i}\right\} $. We define $T_{n}=\sum_{i=1}^{n}\tau_{i}U_{i}$
to be the epochs of our random walk. Let $\left\{ J_{i}\right\} $
be i.i.d r.vs in $\mathbb{Z}^{d}$ and $S_{n}=\sum_{i=1}^{n}J_{i}$
be a discrete random walk on $\mathbb{Z}^{d}$ s.t 
\[
\mathbb{P}\left(J_{n+1}=y|\tau_{n}=t,U_{i+1}=s\right)=p_{t}\left(s,x;y\right).
\]
 Then, conditioning on $\left\{ \tau_{1}=t_{1},\tau_{2}=t_{2},...\right\} $
we define 
\[
X_{t}^{\lyxmathsym{\mbox{II}}}=S_{n}\qquad T_{n}\leq t<T_{n+1}.
\]
 We note that in general $X_{t}^{\lyxmathsym{\mbox{II}}}$ is not
a Markov process, however, if $U_{1}$ is exponentially distributed,
(\ref{eq:transition_probability}) is independent of $s$ and $N_{t}$
counts the number of jumps of $X_{t}^{\lyxmathsym{\mbox{II}}}$ until
time $t$, then $\left(X_{t},N_{t}\right)$ is a Markov process with
the generator 
\begin{equation}
Lf\left(x,z\right)=\sum_{y\sim x}\tau_{z}^{-1}p_{\tau_{z}}\left(y\right)\left(f\left(y,z+1\right)-f\left(x,z\right)\right).\label{eq:quenched_generator}
\end{equation}
We shall refer to $X_{t}^{\lyxmathsym{\mbox{II}}}$ as the quenched
process of a RWRE of Type $\text{\mbox{II}}$. We define the annealed
process of a RWRE of Type $\text{\mbox{II}}$ similarly to that of
type \mbox{I}. That is 
\[
\mathbb{P}\left(\tilde{X}_{\cdot}^{\lyxmathsym{\mbox{II}}}\in dx_{\cdot}\right)=\int\mathbb{P}\left(X_{\cdot}^{\lyxmathsym{\mbox{II}}}\in dx_{\cdot}\right)P_{\boldsymbol{\tau}}\left(d\boldsymbol{\tau}\right),
\]
where $P_{\boldsymbol{\tau}}\left(d\boldsymbol{\tau}\right)$ is a
probability distribution on the Borel sigma-algebra with respect to
the product topology on $\mathbb{R}_{+}^{\mathbb{N}}$ s.t for every
cylinder set of the form $A=\mathbb{R}_{+}\cdots\times\mathbb{R}_{+}\times A_{n_{1}}\times A_{n_{2}}\cdots\times A_{n_{m}}\times\mathbb{R}_{+}\cdots$
with $A_{n_{i}}\subset\mathbb{R}_{+}$, 
\[
\int_{A}P_{\boldsymbol{\tau}}\left(d\boldsymbol{\tau}\right)=\mathbb{P}\left(\tau_{n_{1}}\in A_{n_{1}},\tau_{n_{1}}\in A_{n_{2}},...,\tau_{n_{m}}\in A_{n_{m}}\right).
\]
Here we shall be interested in the limit 
\begin{equation}
n^{-1}\tilde{X}_{tn^{\frac{2}{\alpha}}}^{\lyxmathsym{\mbox{II}}}\Rightarrow\bar{X}_{t}^{\text{\mbox{II}}}\label{eq:scaling_CTRW-2}
\end{equation}
 \\
Let $\mathfrak{M}$ be the set of probability measures whose all moments
are finite. Consider the sets
\begin{align*}
\mathcal{A} & =\cup_{\psi\in\mathfrak{B}_{s^{\alpha}}}\Phi_{\psi}\left(\mathfrak{M}\right),\\
\mathcal{B} & =\Phi_{s^{\alpha}}\left(\mathfrak{M}\right).
\end{align*}
 We have seen already in \lemref{Pseudo_smooth_dense} that $\mathcal{A}$
is weakly dense in $DOA\left(\alpha\right)$. Since $\mathfrak{M}$
is dense in $\mathfrak{L}$ and $\Phi_{\psi}$ is weakly continuous
for every $\psi\in\mathfrak{B}$, we conclude that $\mathcal{B}$
is weakly dense in $\mathfrak{\mathfrak{L}}_{s}^{s^{\alpha}}$. \\
 In order to facilitate the exposition of our results me make the
following technical assumption.
\begin{assumption}
\label{assu:transition probability}Assume $\left\{ J_{i},W_{i}\right\} \in\mathbb{R}^{d}\times\mathbb{R}_{+}$
are i.i.d space-time jumps. We assume that the conditional distribution
$p\left(dx;w\right)=\mathbb{P}\left(J_{1}\in d\boldsymbol{x}|W_{1}=w\right)$
is weakly continuous in $w$, $\mathbb{E}\left(J_{1}|W_{1}=w\right)=\boldsymbol{0}$
, $\boldsymbol{\sigma}^{2}\left(w\right)=\mathbb{E}\left(J_{i}^{T}J_{i}|W_{1}=w\right)$
is a full rank $d\times d$ matrix for $\mathbb{P}\left(W_{1}\in dw\right)$
almost every $w$ and 
\begin{equation}
\sup_{w}\left\Vert \boldsymbol{\sigma}^{2}\left(w\right)\right\Vert <\infty,\label{eq:condition on sigma}
\end{equation}
where $\left\Vert \cdot\right\Vert $ is any norm on the space of
$d\times d$ matrices. 
\end{assumption}
Define 
\[
\boldsymbol{\sigma}_{\mu}^{2}=\int_{\mathbb{R}_{+}}\boldsymbol{\sigma}^{2}\left(t\right)\mu\left(dt\right)\qquad\mu\in\mathfrak{L}.
\]
Suppose $f_{t}\in\mathbb{D}$, we denote by $f_{t}^{s}$ the time-shift
of $f$, i.e.
\[
f_{t}^{s}:=f_{t+s}\qquad t>0.
\]
Recall the definition of $\mu_{f}\left(dx\right)$ in (\ref{eq:mu_f}). 
\begin{thm}
\label{thm:CTRW_as_RWRE}Let $X_{t}$ be a CTRW associated with the
i.i.d space-time jumps $\left(J_{i},W_{i}\right)$ satisfying Assumption
\ref{assu:transition probability} where $W_{1}\in\mathcal{A}$ and
$J_{i}\in\mathbb{Z}^{d}$ . Then $X_{t}$ is an annealed process of
RWRE of type \mbox{I}. Moreover, if $W_{1}$ is also in $\mathcal{B}$
then $X_{t}$ is also an annealed process of RWRE of type \mbox{II}.
In both cases, the limits \ref{eq:scaling_CTRW-1} and \ref{eq:scaling_CTRW-2}
exist and equal 
\begin{equation}
\bar{X}_{t}=B_{E_{t}},\label{eq:CTRW_as_RWRE-3}
\end{equation}
where $E_{t}$ is the inverse of a stable subordinator, and conditioned
on $E_{\cdot}=\xi$, $B_{t}$ is a time-inhomogeneous diffusion whose
generator is 
\begin{equation}
L_{t}\left(f\right)\left(x\right)=\frac{1}{2}\nabla_{x}f^{T}\boldsymbol{\sigma}_{\mu_{\left(\xi^{-1}\right)^{t}}}^{2}\nabla_{x}f,\label{eq:generator of B_t conditioned on E}
\end{equation}
where $\mu\in\mathfrak{L}_{s}$. 
\end{thm}
\begin{proof}
If $W_{1}\in\mathcal{A}$, by \propref{almost_every_CTRW_is_time_changed_CTRW}
and the definition of $\mathcal{A}$, one can find $U_{1}\in\mathfrak{M}$
and an inverse-subordinator $E_{t}^{1}$ s.t 
\begin{equation}
X_{t}\overset{J_{1}}{\sim}\left(\mathring{X}_{E_{t}^{1}-}\right)^{+},\label{eq:CTRW_as_RWRE-1}
\end{equation}
where $\mathring{X}_{t}$ is a CTRW with space-time jumps $\left(J_{i},U_{i}\right)$
with $\mathbb{E}\left(U_{1}\right)<\infty$. Considering (\ref{eq:pseudo_self_similarity})
we see that we may assume that $\mathbb{E}\left(U_{1}\right)=1$ as
this would only change the convergence to a standard stable subordinator
by a constant time change. This proves that $X_{t}$ is an annealed
process of RWRE of type \mbox{I}. Let $X_{t}^{n}$ be the CTRW associated
with the space-time jumps $\left(n^{-1}J_{i},n^{-\frac{2}{\alpha}}W_{i}\right)$,
then 
\[
X_{t}^{n}\overset{J_{1}}{\sim}n^{-1}X_{tn^{\frac{2}{\alpha}}},
\]
 and by \propref{almost_every_CTRW_is_time_changed_CTRW} we may assume
w.l.o.g that there exists a sequence of inverses of subordinators
$E_{t}^{n}$ s.t 
\begin{equation}
E_{t}^{n}\overset{J_{1}}{\rightarrow}E_{t}\label{eq:CTRW_as_RWRE-2}
\end{equation}
 a.s. where $E_{t}$ is the inverse of a stable subrodinator of index
$\alpha$. By \propref{almost_every_CTRW_is_time_changed_CTRW} we
have
\[
X_{t}^{n}\overset{J_{1}}{\sim}\left(\mathring{X}_{E_{t}^{n}-}^{n}\right)^{+},
\]
where $\mathring{X}_{t}^{n}$ is the CTRW associated with $\left\{ n^{-1}J_{i},n^{-2}U_{i}\right\} $.
We now wish to find $\mathbb{P}\left(\left(n^{-1}J_{i+1},n^{-2}U_{i+1}\right)\in\left(dx,du\right)|E_{\cdot}^{n}=\xi_{\cdot}^{n}\right)$.
Let $T_{n}=\sum_{i=1}^{n}U_{i}$ , we have 
\begin{align*}
\text{\ensuremath{\mathbb{P}}}\left(\left(n^{-1}J_{i+1},n^{-\frac{2}{\alpha}}W_{i+1}\right)\in\left(dx,dw\right)\right) & =\\
\text{\ensuremath{\mathbb{P}}}\left(\left(n^{-1}J_{i+1},D_{\left(n^{-2}U_{i+1}+n^{-2}T_{i}\right)}^{n}-D_{n^{-2}T_{i}}^{n}\right)\in\left(dx,dw\right)\right) & =\\
\int\text{\ensuremath{\mathbb{P}}}\left(n^{-1}J_{i+1}\in dx|n^{-2}T_{i}=t,n^{-2}U_{i+1}=u,D_{\cdot}^{n}=\left(\xi_{\cdot}^{n}\right)^{-1}\right)\\
\times\mathbb{P}\left(n^{-2}U_{i+1}\in du\right)\mathbb{P}\left(n^{-2}T_{i}\in dt\right)\mathbb{P}\left(D_{\cdot}^{n}\in d\left(\xi_{\cdot}^{n}\right)^{-1}\right).
\end{align*}
We conclude that
\begin{align*}
\text{\ensuremath{\mathbb{P}}}\left(n^{-1}J_{i+1}\in dx|n^{-2}T_{i}=t,n^{-2}U_{i+1}=u,D_{\cdot}^{n}=\left(\xi_{\cdot}^{n}\right)^{-1}\right) & =p\left(ndx;\left(\xi^{n}\right)_{n^{2}\left(t+u\right)}^{-1}\right).
\end{align*}
 Let $Y_{t}^{n}$ be the Markov process $\left(\left(n^{-1}S_{N_{t}^{n}},n^{-2}T_{N_{t}^{n}}\right)\right)$
conditioned on $\left\{ D_{\cdot}=\left(\xi_{\cdot}^{n}\right)^{-1}\right\} $,
where $S_{n}=\sum_{i=1}^{n}J_{i}$ and $N_{t}^{n}$ is a homogeneous
Poisson process with intensity $n^{2}$. $Y_{t}^{n}$ is a Markov
process with generator 
\[
L^{n}\left(f\right)\left(x,t\right)=n^{2}\int p\left(dy;\left(\xi^{n}\right)_{\left(t+u\right)}^{-1}\right)\mathbb{P}\left(U_{1}\in du\right)\left(f\left(x+yn^{-1},t+un^{-2}\right)-f\left(x,t\right)\right),
\]
 for every $f\in C^{2,1}\left(\mathbb{R}^{d}\times\mathbb{R}_{+}\right)$.
Let $\xi^{-1}=\lim_{n\rightarrow\infty}\left(\xi^{n}\right)^{-1}$
where the limit is in $J_{1}-$topology. If we denote $\mu\left(du\right)=\mathbb{P}\left(U\in du\right)$,
it is not hard to see that $\mu_{\left(\left(\xi^{n}\right)^{-1}\right)^{t}}\rightarrow\text{\ensuremath{\mu_{\left(\xi^{-1}\right)^{t}}} }$
for every $\ensuremath{t\geq0}$ where convergence is in the weak
topology of measures in $\mathfrak{L}$ and where $\mu_{\left(\xi^{-1}\right)^{t}}$
is as in (\ref{eq:mu_f}). By Assumption \ref{assu:transition probability}
it is also not hard to verify that 
\[
L^{n}\left(f\right)\left(x,t\right)\rightarrow L\left(f\right)\left(x,t\right),
\]
where
\begin{equation}
L\left(f\right)\left(x,t\right)=\frac{1}{2}\nabla_{x}f^{T}\boldsymbol{\sigma}_{\mu_{\left(\xi^{-1}\right)^{t}}}^{2}\nabla_{x}f+\frac{\partial}{\partial t}f\label{eq:CTRW_as_RWRE}
\end{equation}
with $\nabla_{x}f^{T}=\left(\frac{\partial}{\partial x_{1}}f,...,\frac{\partial}{\partial x_{d}}f\right)$.
(\ref{eq:condition on sigma}) ensures that (\ref{eq:CTRW_as_RWRE})
is indeed the generator of a Markov process on $\mathbb{D}$ (see
\cite[Theorem 5.4.2]{Kolokoltsov2011}). It follows that $Y_{t}^{n}\overset{J_{1}}{\Rightarrow}Y_{t}$
where $Y_{t}$ is a Markov process whose generator is given by (\ref{eq:CTRW_as_RWRE}).
By Lemma \ref{lem:Straka_and_Henry} we see that 
\[
\mathring{X}_{t}^{n}\overset{J_{1}}{\Rightarrow}B_{t},
\]
where $B_{t}$ is a diffusion with the generator in (\ref{eq:generator of B_t conditioned on E}).
Finally we conclude that 
\[
\mathring{X}_{E_{t}^{n}}^{n}\overset{J_{1}}{\Rightarrow}\left(B_{\xi-}\right)^{+}.
\]
Since the generator in (\ref{eq:generator of B_t conditioned on E})
is a local operator we conclude that $t\mapsto B_{t}$ is continuous
a.s, and that (\ref{eq:CTRW_as_RWRE-3}) holds. Next we assume that
$W_{1}\in\mathfrak{\mathfrak{L}}_{s}^{s^{\alpha}}$. Note that this
suggests that $E_{t}^{n}=E_{t}$ for every $n\geq1$ and that 
\begin{align}
W_{i} & \sim D_{\left(U_{i}+T_{i-1}\right)}-D_{T_{i-1}}\nonumber \\
 & \sim U_{i}^{\frac{1}{\alpha}}D_{1}.\label{eq:Type_2}
\end{align}
The mapping $U\mapsto U^{\frac{1}{\alpha}}$ maps the set $\mathfrak{M}$
onto $\mathfrak{M}$. It follows that $X_{t}$ is the CTRW associated
with the space-time jumps $\left\{ J_{i},U_{i}^{\frac{1}{\alpha}}\tau_{i}\right\} $
with $\boldsymbol{\tau}=\left\{ \tau_{i}\right\} $ where $\tau_{1}\sim D_{1}$.
This shows that $X_{t}$ is an annealed process of RWRE of type \mbox{II}
with waiting times $U_{i}^{'}=\left(\mathbb{E}\left(U_{i}^{\frac{1}{\alpha}}\right)\right)^{-1}U_{i}^{\frac{1}{\alpha}}$
and random environment $\boldsymbol{\tau}'=\left\{ \mathbb{E}\left(U_{i}^{\frac{1}{\alpha}}\right)\tau_{i}\right\} $.
Assume for simplicity that $\mathbb{E}\left(U_{i}^{\frac{1}{\alpha}}\right)=1$.
Using (\ref{eq:Type_2}) and the calculations for the RWRE of type
\mbox{I} we conclude that the quenched limit of the RWRE of type \mbox{II}
is $B_{\xi}$. 
\end{proof}

\section{bound on the error}

In this section we give a polynomial bound on the distance between
the law of a given uncoupled CTRW $Y_{t}$ and the law of a time changed
CTRW $\mathring{X_{E_{t}^{n}}^{n}}$ on the space $\mathbb{D}[0,T]$.
The proof is constructive and therefore provides us with the space-time
jumps of $\mathring{X}_{t}$ as well as with the inverse-subordinators
$E_{t}^{n}$. The bound relies on the following lemma. 
\begin{lem}
 \label{lem:annoying_lemma}Let $X\in\mathfrak{L}_{s^{\alpha}}$
be a r.v. with tail $\bar{f}\left(t\right)=\mathbb{P}\left(X>t\right)$.
There exists a r.v $Y\in\mathfrak{L}_{s}^{\overline{s^{\alpha}}}$
and a coupling $\mathbb{P}_{couple}$ of $X$ and $Y$ s.t
\[
\mathbb{P}_{couple}\left(\left|X-Y\right|>t\right)=o\left(\bar{f}\left(t\right)\right)
\]
 
\end{lem}
\begin{proof}
Suppose $X$ is a r.v in $\mathfrak{L}_{s^{\alpha}}$ . It follows
that there exists a function $L\left(t\right)$ which is positive
and slowly varying s.t $\text{\ensuremath{\mathbb{P}}}\left(X\geq t\right)=L\left(t\right)t^{-\alpha}$
. By \lemref{Pseudo_smooth_dense} and (\ref{eq:Pseudo-smooth dense-1})
we see that there exists $Y\in\mathfrak{L}_{s}^{\overline{s^{\alpha}}}$
s.t $\mathbb{P}\left(Y\geq t\right)=L\left(t\right)t^{-\alpha}+g\left(t\right)$,
where $g\left(t\right)=o\left(L\left(t\right)t^{-\alpha}\right)$.
We denote $F_{1}\left(t\right)=\mathbb{P}\left(Y\geq t\right)$ ,
$F_{2}\left(t\right)=\mathbb{P}\left(X\geq t\right)$ and $I_{j}=[j,j+1)$
for $j\in\mathbb{Z}_{+}$. We begin by coupling $X$ and $Y$ in any
way on $I_{j}$, note that the mass that can be coupled on $I_{j}$
is $\min\left\{ F_{1}\left(I_{j}\right),F_{2}\left(I_{j}\right)\right\} $
where $F_{i}\left(I_{j}\right)=F_{i}\left(j\right)-F_{i}\left(j+1\right)$
for $i\in\left\{ 1,2\right\} $, and the mass that is excessive and
could not be coupled is $\left|F_{1}\left(I_{j}\right)-F_{2}\left(I_{j}\right)\right|=\left|g\left(j\right)-g\left(j+1\right)\right|$.
Note also that the sign of $g\left(j\right)-g\left(j+1\right)$ determines
whether $F_{1}\left(I_{j}\right)>F_{2}\left(I_{j}\right)$ $\left(g\left(j\right)-g\left(j+1\right)>0\right)$,
$F_{1}\left(I_{j}\right)<F_{2}\left(I_{j}\right)$ $\left(g\left(j\right)-g\left(j+1\right)<0\right)$
and $F_{1}\left(I_{j}\right)=F_{2}\left(I_{j}\right)\left(g\left(j\right)-g\left(j+1\right)=0\right)$.
Next we couple the excessive mass of $X$ and $Y$ on each interval
of the form $[2^{n},2^{n+1})$ in the following way: let $\left\{ I_{i_{k}}\right\} _{k=1}^{m_{q}}$
and $\left\{ I_{j_{k}}\right\} _{k=1}^{m_{s}}$ be the sets of intervals
whose excessive mass from the partial coupling before is negative
and non-negative respectively. More precisely, let $I_{j}\subset[2^{n},2^{n+1})$
then $I_{j}\in\left\{ I_{i_{k}}\right\} _{k=1}^{m_{q}}$ ($I_{j}\in\left\{ I_{j_{k}}\right\} _{k=1}^{m_{s}}$)
iff $g\left(j\right)-g\left(j+1\right)<0$ ($g\left(j\right)-g\left(j+1\right)\geq0$).
So $\left\{ \cup_{k=1}^{m_{q}}I_{i_{k}}\right\} \cup\left\{ \cup_{k=1}^{m_{s}}I_{j_{k}}\right\} =[2^{n},2^{n+1})$
. Imagine that each $I_{i_{k}}$ is a customer with negative mass
$q_{k}$ and each $I_{j_{k}}$is a server with positive mass $s_{k}$.
Customers enter the queue according to their original order in $[2^{n},2^{n+1})$,
that is , $I_{i_{k_{1}}}$ is in front of $I_{i_{k_{2}}}$ iff $i_{k_{1}}<i_{k_{2}}$.
The customer $I_{i_{k}}$leaves the queue only after he was served
by $m$ servers whose total mass is at least $q_{k}$. Server $I_{j_{k}}$
leaves the line as soon as he has served all its mass. For example,
if in the interval $[4,8)$ we have the following
\begin{equation}
-\overset{I_{4}}{0.2},-\overset{I_{5}}{0.4},\overset{I_{6}}{0.1},\overset{I_{7}}{0.7}.\label{eq:annoying lemma-4}
\end{equation}
In this case, $I_{i_{1}}=[4,5),I_{i_{2}}=[5,6)$ and $I_{j_{1}}=[6,7),I_{j_{2}}=[7,8)$.
then the coupling will be 
\begin{align*}
-0.4,-0.2 & |0.1,0.7\\
-0.4,-0.1 & |0.7\\
-0.4 & |0.6\\
 & 0.2,
\end{align*}
and so $g\left(4\right)-g\left(8\right)=0.2$, which is the excessive
mass of $F_{1}\left([4,8)\right)$ over $F_{2}\left([4,8)\right)$
that can not be coupled in the interval $[4,8)$. We say that the
interval $I_{i_{k}}$is $i$-bad if the last server $I_{j_{k'}}$
that served him is such that $|i_{k}-j_{k'}|>i$. For example, in
(\ref{eq:annoying lemma-4}) the customer $I_{4}$ was served by both
$I_{6}$ and $I_{7}$ and since $7-4=3<4$ it is not $4$-bad but
is $2$-bad. Note that if $I_{j}\in[2^{n},2^{n+1})$ then $I_{j}$
is $i$-bad iff one of the following conditions is satisfied
\begin{align*}
F_{1}\left([2^{n},j-i)\right) & \geq F_{2}\left([2^{n},j+1)\right)\\
F_{1}\left([2^{n},j+i+1)\right) & <F_{2}\left([2^{n},j+1)\right).
\end{align*}
Define 
\begin{align}
\epsilon_{i}^{1} & =\sup_{j\geq i}\sup_{1\leq\lambda\leq2}\left|\frac{L\left(j\lambda\right)}{L\left(j\right)}-1\right|\label{eq:annoying lemma}\\
\epsilon_{i}^{2} & =\sup_{t\geq i}\left|\frac{g\left(t\right)}{L\left(t\right)t^{-\alpha}}\right|.\label{eq:annoying lemma-1}
\end{align}
Note that by the UCT and the definition of $g\left(t\right)$ $\epsilon_{i}^{1},\epsilon_{i}^{2}\overset{i\rightarrow\infty}{\rightarrow}0$.
Fix a positive integer $i$. Note that potential $i$- bad intervals
$I_{j}$ should be looked for for $j\geq2^{[\log i]+1}$, where throughout
the proof we use $\log x=\log_{2}x$. Let us now check the two conditions.
Let $t=2^{[\log j]}$, then condition one is 
\begin{align}
F_{1}\left(t\right)-F_{1}\left(j-i\right) & \geq F_{2}\left(t\right)-F_{2}\left(j+1\right)\nonumber \\
F_{2}\left(t\right)-F_{2}\left(j-i\right) & +g\left(t\right)-g\left(j-i\right)\geq F_{2}\left(t\right)-F_{2}\left(j+1\right)\nonumber \\
F_{2}\left(j+1\right)-F_{2}\left(j-i\right) & \geq g\left(j-i\right)-g\left(t\right).\label{eq:annoying lemma-8}
\end{align}
Note that by (\ref{eq:annoying lemma-1}) it is enough to look for
$j$'s that satisfy 
\[
L\left(j+1\right)\left(j+1\right)^{-\alpha}-L\left(j-i\right)\left(j-i\right)^{-\alpha}\geq-2\epsilon_{t}^{2}L\left(t\right)t^{-\alpha}.
\]
If $L_{max}=\sup_{t\leq y\leq2t}\left|L\left(y\right)\right|$, by
(\ref{eq:annoying lemma}) we can look for $j$'s that satisfy
\[
L_{max}\left(\left(j+1\right)^{-\alpha}-\left(1-\epsilon_{t}^{1}\right)\left(j-i\right)^{-\alpha}\right)\geq-2\epsilon_{t}^{2}t^{-\alpha}L_{max}.
\]
Using the convexity of $t\mapsto t^{-\alpha}$ we may consider 
\begin{align*}
-\alpha\left(j+1\right)^{-\alpha-1}\left(i+1\right)+\epsilon_{t}^{1}\left(j-i\right)^{-\alpha} & \geq-2\epsilon_{t}^{2}t^{-\alpha},
\end{align*}
or
\[
j\geq t^{\frac{\alpha}{1+\alpha}}\left(\left(i+1\right)\alpha\right)^{\frac{1}{\alpha+1}}\left(2\epsilon_{t}^{2}+\epsilon_{t}^{1}\right)^{-\frac{1}{1+\alpha}}-1.
\]
Note that $t$ is at least $2^{[\log\left(i\right)]}$ and so 
\begin{equation}
j\geq\left(\left(i+1\right)\alpha\right)^{\frac{1}{\alpha+1}}2^{[\log\left(i\right)]}{}^{\frac{\alpha}{1+\alpha}}\left(2\epsilon_{t}^{2}+\epsilon_{t}^{1}\right)^{-\frac{1}{1+\alpha}}-1.\label{eq:annoying lemma-2}
\end{equation}
It follows that for a fixed $i$ , $i$-bad $j$'s who satisfy the
first condition should be looked for above a number that increases
super-linearly with $i$. Similarly, for the second condition we obtain
the following condition 
\begin{equation}
j>\left(i\alpha\right)^{\frac{1}{\alpha+1}}2^{[\log\left(i\right)]}{}^{\frac{\alpha}{1+\alpha}}\left(2\epsilon_{t}^{2}+\epsilon_{t}^{1}\right)^{-\frac{1}{1+\alpha}}-1.\label{eq:annoying lemma-3}
\end{equation}
Let us denote by $ic_{t}$ ($t=2^{\left[\log i\right]}$) the r.h.s
of (\ref{eq:annoying lemma-3}). It follows that one cannot find $i$-bad
$j$'s between $i$ and $ic_{t}$ where the latter increases super-linearly
in $i$. Let $W=\left|X-Y\right|$ be the absolute difference between
$X$ and $Y$ in our coupled space $\left(\Omega_{couple},\mathcal{F}_{couple},\mathbb{P}_{couple}\right)$.
If $I_{j}$ is not $i$-bad and was coupled in the second stage, then
$\left\{ X\in I_{j}\right\} \subset\left\{ W\leq i\right\} $ for
$i>1$. Since on each interval of the form $[2^{n},2^{n+1})$, for
$n\geq\log\left(ic_{t}\right)$, we coupled the r.v in such a way
that it has no $i$-bad intervals, the only mass that may affect the
event $\left\{ W>i\right\} $ is $\left|g\left(2^{n}\right)-g\left(2^{n+1}\right)\right|$.
It follows that 
\begin{equation}
\mathbb{P}_{couple}\left(W>i\right)\leq\sum_{k=\left[\log\left(i\right)\right]}^{\left[\log\left(ic_{t}\right)\right]-1}\left|g\left(2^{k}\right)-g\left(2^{k+1}\right)\right|+L\left(\frac{ic_{t}}{4}\right)\left(\frac{ic_{t}}{4}\right)^{-\alpha}+\left|g\left(\frac{ic_{t}}{4}\right)\right|.\label{eq:annoying lemma-5}
\end{equation}
We claim now that 
\begin{equation}
\sum_{k=\left[\log\left(i\right)\right]}^{\left[\log\left(ic_{t}\right)\right]-1}\left|g\left(2^{k}\right)-g\left(2^{k+1}\right)\right|=o\left(L\left(i\right)i^{-\alpha}\right).\label{eq:annoying lemma-6}
\end{equation}
To see that we note that by (\ref{eq:annoying lemma-1}) and the UCT,
for any $C>\frac{1+2^{\alpha}}{1-2^{\alpha}}$and large enough $i$
we have
\begin{align*}
\left|g\left(2^{k}\right)-g\left(2^{k+1}\right)\right| & \leq\epsilon_{t}^{2}2^{-\alpha k}L\left(2^{k}\right)\left(1+2^{-\alpha}\frac{L\left(2^{k+1}\right)}{L\left(2^{k}\right)}\right)\\
 & \leq C\epsilon_{t}^{2}2^{-\alpha k}L\left(2^{k}\right)\left(1-2^{-\alpha}\frac{L\left(2^{k+1}\right)}{L\left(2^{k}\right)}\right),
\end{align*}
It follows that for large enough $i$ we have
\begin{align*}
\sum_{k=\left[\log\left(i\right)\right]}^{\left[\log\left(ic_{t}\right)\right]-1}\left|g\left(2^{k}\right)-g\left(2^{k+1}\right)\right| & \leq\epsilon_{i}^{2}C\sum_{k=\left[\log\left(i\right)\right]}^{\left[\log\left(ic_{t}\right)\right]-1}F_{2}\left(2^{k}\right)-F_{2}\left(2^{k+1}\right)\\
 & \leq\epsilon_{i}^{2}C\left(L\left(2^{\left[\log\left(i\right)\right]}\right)2^{-\left[\log\left(i\right)\right]\alpha}-L\left(ic_{t}\right)2^{-ic_{t}\alpha}\right),
\end{align*}
and (\ref{eq:annoying lemma-6}) is implied. It follows form (\ref{eq:annoying lemma-5})
and (\ref{eq:annoying lemma-6}) that
\begin{equation}
\mathbb{P}_{couple}\left(W>i\right)=o\left(L\left(i\right)i^{-\alpha}\right),\label{eq:annoying lemma-7}
\end{equation}
 and it is straightforward to see that (\ref{eq:annoying lemma-7})
holds when $i\in\mathbb{R}_{+}$ .  
\end{proof}
In the following result we limit ourselves to case where the waiting
times of $Y_{t}^{1}$ is such that $\mathbb{P}\left(W_{i}>t\right)\sim ct^{-\alpha}$,
where $c$ is some positive constant. This assumption is important
for the result. 
\begin{thm}
\label{thm:quantitive_result}Let $Y_{t}^{n}$ be the CTRW associated
with the i.i.d space-time jumps $\left\{ n^{-\frac{1}{\alpha}}W_{i},n^{-\frac{1}{2}}J_{i}\right\} $
where $\mathbb{P}\left(W_{1}>t\right)=\left[\Gamma\left(1-\alpha\right)\right]^{-1}t^{-\alpha}+g_{2}\left(t\right)$
and $g_{2}\left(t\right)=O\left(t^{-\beta}\right)$ for $\beta>\alpha$
and $J_{1}\in\mathbb{R}$ has variance $1$ and zero mean. Then there
exists a CTRW $\mathring{X}_{t}^{n}$ associated with the i.i.d space-time
jumps $\left(n^{-\frac{1}{2}}J_{i},n^{-1}\mathring{U}_{i}\right)$
s.t $\mathring{U}_{1}$ has finite mean, a sequence of inverse subordinators
$E_{t}^{n}$ so that for every $c<\xi_{0}$,
\[
\rho_{J_{1}}\left(Y_{t}^{n},\mathring{X}_{E_{t}^{n}}^{n}\right)<Cn^{-c},
\]
where $\xi_{0}=\min\left\{ \frac{\alpha}{7\alpha+4},\frac{\beta-\alpha}{3\beta+\alpha+4}\right\} $.
\end{thm}
\begin{proof}
~
\begin{enumerate}[label=Step \arabic{enumi}]
\item \label{enu:preperation} First consider for every $n$ the sequence
$\left\{ a_{n}W_{i}\right\} _{i=1}^{\infty}$. If $\mathbb{P}\left(W_{1}>t\right)=L\left(t\right)t^{-\alpha}$,
by \lemref{Pseudo_smooth_dense} we can approximate $W_{1}$ by a
distribution $U\in\mathfrak{L}_{s}^{\overline{s^{\alpha}}}$ s.t $\mathbb{P}\left(U>t\right)\sim L\left(t\right)t^{-\alpha}$.
Let $\left\{ U_{i}\right\} $ be a sequence of i.i.d r.v's s.t $U_{1}\sim U$.
Define $X_{t}^{n}$ to be the CTRW associated with the space time
jumps $\left\{ n^{-\frac{1}{2}}J_{i},n^{-\frac{1}{\alpha}}U_{i}\right\} $
We wish to construct a set $A\in\Omega_{couple}$ of probability larger
than $1-\epsilon$ on which we can bound the distance ($d_{J_{1}}$)
between two trajectories of the processes $X_{t}^{n}$ and $Y_{t}^{n}$.
In order to use Lemma \ref{lem:annoying_lemma} we must limit our
discussion to finite number of jumps by time $T$. We shall use the
fact that for every coupling $p_{X,Y}$ of some r.vs $X$ and $Y$
, if $p_{X}\left(X\in A\right)<\frac{\epsilon}{4}$ and $p_{Y}\left(Y\in B\right)<\frac{\epsilon}{4}$
then $p_{X,Y}\left(\left|X-Y\right|1_{\left\{ X\in A\right\} \cup\left\{ Y\in B\right\} }>\epsilon\right)<\frac{\epsilon}{2}$,
for any coupling $p_{X,Y}$ of $p_{X}$ and $p_{Y}$. And so, if we
show that 
\[
p_{X,Y}\left(\left|X-Y\right|1_{\left\{ X\in A^{c}\right\} \cap\left\{ Y\in B^{c}\right\} }>\epsilon\right)<\frac{\epsilon}{2},
\]
we see that
\begin{align*}
p_{X,Y}\left(\left|X-Y\right|>\epsilon\right) & =\\
p_{X,Y}\left(\left|X-Y\right|1_{\left\{ X\in A\right\} \cup\left\{ Y\in B\right\} }+\left|X-Y\right|1_{\left\{ X\in A^{c}\right\} \cap\left\{ Y\in B^{c}\right\} }>\epsilon\right) & <\epsilon.
\end{align*}
 Suppose there exists a sequence $M_{1}\left(n\right)\rightarrow\infty$
s.t for large enough $n$
\begin{alignat*}{1}
\mathbb{P}\left(\sum_{i=1}^{M_{1}\left(n\right)}a_{n}W_{i}\leq T\right) & \leq\frac{\epsilon}{8},\\
\mathbb{P}\left(\sum_{i=1}^{M_{1}\left(n\right)}a_{n}U_{i}\leq T\right) & \leq\frac{\epsilon}{8}.
\end{alignat*}
Next assume there exists a sequence $M_{2}\left(n\right)\rightarrow\infty$
s.t for large enough $n$
\begin{align}
\mathbb{P}\left(\sum_{i=1}^{M_{2}\left(n\right)}a_{n}W_{i}\leq\frac{\epsilon}{2}\right) & <\frac{\epsilon}{8},\label{eq:main thm-1}\\
\mathbb{P}\left(\sum_{i=1}^{M_{2}\left(n\right)}a_{n}U_{i}\leq\frac{\epsilon}{2}\right) & <\frac{\epsilon}{8}.
\end{align}
Moreover, assume that for large enough $n$
\[
\mathbb{P}\left(\overline{S}_{\left\{ 0,1,...,M_{2}\left(n\right)\right\} }>\frac{\epsilon}{4}\right)<\frac{\epsilon}{8},
\]
where for a set $A\subset\mathbb{Z}_{+}$, with $j_{i}=\inf A$, 
\[
\overline{S}_{A}=\sup_{i\in A}\left|\sum_{j=j_{i}}^{i}n^{-\frac{1}{2}}J_{i}\right|.
\]
Define the random sets
\begin{align*}
B_{n}^{Y}[a,b] & =\left\{ j:\sum_{i=1}^{j}a_{n}W_{i}:\in[a,b]\right\} \\
B_{n}^{X}[a,b] & =\left\{ j:\sum_{i=1}^{j}a_{n}U_{i}:\in[a,b]\right\} ,
\end{align*}
that is, $B_{n}^{Y}[a,b]$ is the set of the indices of the jumps
that occurred in the time interval $[a,b]$. Also define $i_{0}^{Y}=\inf B_{n}^{Y}[a,b]$
and $i_{0}^{X}=\inf B_{n}^{X}[a,b]$. By \lemref{annoying_lemma}
we know that we can construct a probability space $\left(\Omega_{couple},\mathcal{F}_{couple},\mathbb{P}_{couple}\right)$
on which one can find the sequence $\left\{ W_{i}\right\} $ and $\left\{ U_{i}\right\} $
s.t for large enough $n$ we have 
\begin{equation}
\mathbb{P}_{couple}\left(\sum_{i=1}^{M_{1}\left(n\right)}a_{n}\left|W_{i}-U_{i}\right|>\frac{\epsilon}{2}\right)<\frac{\epsilon}{4}.\label{eq:main thm-4}
\end{equation}
It is implied that for $N_{0}$ large enough, for every $n>N_{0}$,
one can find a set $A_{n}\in\mathcal{F}_{couple}$ s.t $\mathbb{P}_{couple}\left(A_{n}\right)>1-\epsilon$
and conditioned on $A_{n}$ we have
\begin{align}
\mathbb{P}_{couple}\left(\left|B_{n}^{Y}[0,T]\right|>M_{1}\left(n\right)|A_{n}\right) & =0\nonumber \\
\mathbb{P}_{couple}\left(\left|B_{n}^{Y}[T-\frac{\epsilon}{2},T]\right|>M_{2}\left(n\right)|A_{n}\right) & =0\nonumber \\
\mathbb{P}_{couple}\left(\overline{S}_{B_{n}^{Y}[T-\frac{\epsilon}{2},T]}>\frac{\epsilon}{4}|A_{n}\right) & =0\nonumber \\
\mathbb{P}_{couple}\left(\sum_{i=1}^{M_{1}\left(n\right)}a_{n}\left|W_{i}-U_{i}\right|>\frac{\epsilon}{2}|A_{n}\right) & =0,\label{eq:main thm-5}
\end{align}
where the first three equations in (\ref{eq:main thm-5}) are true
for the sets $B_{n}^{X}[0,T]$ and $B_{n}^{X}[T-r^{-1},T]$ as well. 
\item \label{enu:Step_approximation}Let $M\in\mathbb{Z}_{+}$ and $d_{l_{1}}^{M}:\left(\mathbb{R}\times\mathbb{R}_{+}\right)^{M}\rightarrow\mathbb{R}_{+}$
be the metric on vectors of real numbers defined by 
\[
d_{l_{1}}^{M}\left(\left\{ a_{n}^{1},a_{n}^{2}\right\} ,\left\{ b_{n}^{1},b_{n}^{2}\right\} \right)=\sum_{n=1}^{M}\left|a_{i}^{1}-b_{i}^{1}\right|+\left|a_{i}^{2}-b_{i}^{2}\right|.
\]
 Consider the set $\mathcal{A}=\left\{ \left(J_{i},W_{i}\right)\in\left(\mathbb{R}\times\mathbb{R}_{+}\right)^{M}:\sum_{i=1}^{M}W_{i}\leq T\right\} $
equipped with $d_{l_{1}}^{M},$i.e.
\[
d_{l_{1}}^{M}\left(\left(J_{i}^{1},W_{i}^{1}\right),\left(J_{i}^{2},W_{i}^{2}\right)\right)=\sum_{i=1}^{M}\left|J_{i}^{1}-J_{i}^{2}\right|+\left|W_{i}^{1}-W_{i}^{2}\right|.
\]
Define the mapping $\mathcal{T}:\mathcal{A}\rightarrow\mathbb{D}[0,T]$
by 
\[
\left(J_{i},W_{i}\right)_{i=1}^{M}\mapsto f_{t}=\sum_{i=1}^{M}J_{i}1_{\left\{ \sum_{j=1}^{i}W_{j}\leq t\right\} }.
\]
We claim that $\mathcal{T}:\left(\mathcal{A},d_{l_{1}}^{M}\right)\rightarrow\left(\mathbb{D}[0,T],d_{J_{1}}\right)$
is a contraction. To see that, let $\left(J_{i}^{1},W_{i}^{1}\right),\left(J_{i}^{2},W_{i}^{2}\right)\in\left(\mathbb{R}\times\mathbb{R}_{+}\right)^{M}$,
and define 
\[
\lambda_{t}=\left\{ \begin{array}{cc}
t\frac{W_{1}^{2}}{W_{1}^{1}} & 0\leq t<W_{1}^{1}\\
\left(t-W_{1}^{1}\right)\frac{W_{2}^{2}}{W_{2}^{1}}+W_{1}^{2} & W_{1}^{1}\leq t<W_{1}^{1}+W_{2}^{1}\\
\vdots & \vdots\\
\left(t-\sum_{i=1}^{M-1}W_{i}^{1}\right)\frac{W_{M}^{2}}{W_{M}^{1}}+\sum_{i=1}^{M-1}W_{1}^{2} & \sum_{i=1}^{M-1}W_{i}^{1}\leq t\leq T
\end{array}\right..
\]
Note that 
\[
\left\Vert \mathcal{T}\left[\left(J_{i}^{1},W_{i}^{1}\right)\right]\left(\lambda_{t}\right)-\mathcal{T}\left[\left(J_{i}^{2},W_{i}^{2}\right)\right]\left(t\right)\right\Vert \leq\sup_{i}\left|J_{i}^{1}-J_{i}^{2}\right|\leq\sum_{i=1}^{M}\left|J_{i}^{1}-J_{i}^{2}\right|,
\]
since the regeneration points of $\mathcal{T}\left[\left(J_{i}^{2},W_{i}^{2}\right)\right]\left(t\right)$
and $\mathcal{T}\left[\left(J_{i}^{1},W_{i}^{1}\right)\right]\left(\lambda_{t}\right)$
are the same. Next note that since $\lambda_{t}$ is piece-wise linear
\[
\left\Vert \lambda_{t}-t\right\Vert \leq\sup_{t_{i}}\left|\lambda_{t_{i}}-t_{i}\right|,
\]
where $t_{i}\in\left\{ \sum_{j=1}^{i}W_{j}^{1}:1\leq i\leq M\right\} .$
Or equivalently, 
\[
\left\Vert \lambda_{t}-t\right\Vert =\sup_{1\leq i\leq M}\left|\sum_{j=1}^{i}W_{j}^{2}-\sum_{j=1}^{i}W_{j}^{1}\right|\leq\sum_{i=1}^{M}\left|W_{i}^{1}-W_{i}^{2}\right|.
\]
It follows that 
\begin{align}
 & d_{J_{1}}\left(\mathcal{T}\left[\left(J_{i}^{1},W_{i}^{1}\right)\right]\left(t\right)-\mathcal{T}\left[\left(J_{i}^{2},W_{i}^{2}\right)\right]\left(t\right)\right)\nonumber \\
 & \leq\left\Vert \lambda_{t}-t\right\Vert \wedge\left\Vert \mathcal{T}\left[\left(J_{i}^{1},W_{i}^{1}\right)\right]\left(\lambda_{t}\right)-\mathcal{T}\left[\left(J_{i}^{2},W_{i}^{2}\right)\right]\left(t\right)\right\Vert \nonumber \\
 & \leq d_{l_{1}}\left(\left(J_{i}^{1},W_{i}^{1}\right),\left(J_{i}^{2},W_{i}^{2}\right)\right),\label{eq:main thm-7}
\end{align}
so that $\mathcal{T}$ is indeed a contraction. Next, let $x_{t}^{n}$
and $y_{t}^{n}$ be two realizations of $X_{t}^{n}$ and $Y_{t}^{n}$
respectively on the set $A^{n}$. Suppose w.l.o.g that $x_{t}^{n}$
has at least the same number of jumps by time $T-\frac{\epsilon}{2}$
as $y_{t}^{n}$, that is 
\[
\left|B_{n}^{Y}[0,T-\frac{\epsilon}{2}]\right|\leq\left|B_{n}^{X}[0,T-\frac{\epsilon}{2}]\right|.
\]
By (\ref{eq:main thm-5}) we have
\[
\sum_{i=1}^{M_{1}\left(n\right)}a_{n}\left|W_{i}-U_{i}\right|<\frac{\epsilon}{2},
\]
which implies that one can find $J_{diff}:=\left|B_{n}^{X}[0,T-\frac{\epsilon}{2}]\right|-\left|B_{n}^{Y}[0,T-\frac{\epsilon}{2}]\right|$
jumps of $y_{t}^{n}$ in the interval $[T-\frac{\epsilon}{2},T]$.
Let $\tilde{x}_{t}\in\mathbb{D}[0,T]$ s.t
\[
\tilde{x}_{t}=x_{t}-\sum_{i\in B^{X}[T-\frac{\epsilon}{2},t]}n^{-\frac{1}{2}}J_{i},
\]
where if $t<T-\frac{\epsilon}{2}$ the summation vanishes. In words,
$\tilde{x}_{t}$ equals to $x_{t}$ up to time $T-\frac{\epsilon}{2}$
and equals $x_{t}\left(T-\frac{\epsilon}{2}\right)$ on the interval
$[T-\frac{\epsilon}{2},T]$. Next we define the time $T_{diff}=\inf\left\{ t:\left|B_{n}^{Y}[T-\frac{\epsilon}{2},t]\right|\geq J_{diff}\right\} $
and 
\[
\tilde{y}_{t}=y_{\left(T_{diff}\land t\right)}.
\]
$\tilde{y}_{t}$ stands for the function that equals $y_{t}$ up to
the point where it has jumped the same number of jumps as $\tilde{x}_{t}$.
Next note that on $A^{n}$, 
\begin{align}
d_{J_{1}}\left(x_{t},y_{t}\right) & \leq d_{J_{1}}\left(\tilde{x}_{t},\tilde{y}_{t}\right)+d_{J_{1}}\left(x_{t}-\tilde{x}_{t},y_{t}-\tilde{y}_{t}\right)\nonumber \\
 & \leq d_{l_{1}}^{M}\left(\left(J_{i}^{n}\left(\omega\right),a_{n}W_{i}\left(\omega\right)\right),\left(J_{i}^{n}\left(\omega\right),a_{n}U_{i}\left(\omega\right)\right)\right)\nonumber \\
 & +\overline{S}_{B_{n}^{Y}[T-\frac{\epsilon}{2},T]}+\overline{S}_{B_{n}^{X}[T-\frac{\epsilon}{2},T]}\nonumber \\
 & <\epsilon,\label{eq:main thm-6}
\end{align}
where $\omega\in\Omega_{couple}$ is such that 
\begin{align*}
\mathcal{T}\left(J_{i}^{n}\left(\omega\right),a_{n}W_{i}\left(\omega\right)\right) & =x_{t}\\
\mathcal{T}\left(J_{i}^{n}\left(\omega\right),a_{n}U_{i}\left(\omega\right)\right) & =y_{t}\\
\left|B_{n}^{X}[0,T-\frac{\epsilon}{2}]\right| & =M.
\end{align*}
Inequality (\ref{eq:main thm-6}) follows from (\ref{eq:main thm-5})
and (\ref{eq:main thm-7}). We have showed that there exists a coupling
s.t
\[
\mathbb{P}_{couple}\left(d_{J_{1}}\left(X_{t}^{n},Y_{t}^{n}\right)>\epsilon\right)<\epsilon,
\]
or that 
\[
\rho_{J_{1}}\left(Y_{t}^{n},X_{t}^{n}\right)<\epsilon.
\]
 
\item Let $W$ be a r.v s.t $W\sim W_{1}$. In order to approximate $W$
by elements in $\mathfrak{L}_{s}^{s^{\alpha}}$ we follow the recipe
in Lemma \ref{lem:Pseudo_smooth_dense}. We first introduce $\mathcal{W}_{m}=W1_{[0,m]}$
and 
\[
\mu_{i}^{m}=\mathbb{E}\left(\left(\mathcal{W}_{m}\right)^{i}\right),
\]
the i'th moment of $\mathcal{W}_{m}$. We proceed to defining the
symbol 
\[
\psi\left(s\right)=-s-\left(\mu_{1}^{m}\right)^{-1}\int_{0}^{\infty}\left(e^{-sy}-1\right)\frac{\alpha y^{-\alpha-1}}{\Gamma\left(1-\alpha\right)}dy.
\]
Note that $\psi\left(s\right)$ is the symbol of the subordinator
$t+D_{\nicefrac{t}{\mu_{1}^{m}}}$ , where $D_{t}$ is the standard
stable subordinator of index $0<\alpha<1$ whose LT is $\mathbb{E}\left(e^{-sD_{t}}\right)=e^{-ts^{\alpha}}$.
Note that we somewhat deviate form the recipe in Lemma \ref{lem:Pseudo_smooth_dense},
where we would use the symbol
\begin{equation}
\psi'\left(s\right)=-s-\left(\mu_{1}^{m}\right)^{-1}\int_{0}^{\infty}\left(e^{-sy}-1\right)f\left(dy\right),\label{eq:example-2}
\end{equation}
where $f\left(dy\right)$ is the distribution of $W$. However, since
the purpose of the $f\left(dy\right)$ in Equation (\ref{eq:example-2})
is to obtain a regularity of $s^{\alpha}L\left(s^{-1}\right)$ around
zero for $\psi\left(s\right)$, it is clear that in this case a stable
subordinator would do the job. An expression for the tail of a stable
subordinator at time $t>0$ can be found in \cite[Eq. 2.4.3]{Zolotarev86}
to be(with some algebraic manipulations) 
\[
F_{t}^{D}\left(x\right)=\sum_{n=1}^{\infty}\left(-1\right)^{n-1}\frac{x^{-\alpha n}t^{n}}{\Gamma\left(1-\alpha n\right)n!}\qquad x>0,t>0.
\]
Let $h_{m}\left(dy\right)=\Phi_{\psi}\left(f_{m}\right)$, where $f_{m}\left(dy\right)=\mathbb{P}\left(\mathcal{W}_{m}\in dy\right)$.
We have for $x>m$
\begin{align*}
\bar{h}_{m}\left(x\right) & =\mathbb{P}\left(\mathcal{W}_{m}+D_{\nicefrac{\mathcal{W}_{m}}{\mu_{1}^{m}}}>x\right)\\
 & =\int_{0}^{\infty}F_{\nicefrac{y}{\mu_{1}^{m}}}^{D}\left(x-y\right)f_{m}\left(dy\right).
\end{align*}
Moreover, we see that for $x>m$
\[
\int_{0}^{\infty}F_{\nicefrac{y}{\mu_{1}^{m}}}^{D}\left(x\right)f_{m}\left(dy\right)\leq\bar{h}_{m}\left(x\right)\leq\int_{0}^{\infty}F_{\nicefrac{y}{\mu_{1}^{m}}}^{D}\left(x-m\right)f_{m}\left(dy\right),
\]
which, by the analyticity of $F_{t}^{D}$ and the compact support
of $f_{m}\left(dy\right)$ shows that 
\begin{align}
\frac{x^{-\alpha}}{\Gamma\left(1-\alpha\right)}-\frac{x^{-2\alpha}}{\Gamma\left(1-2\alpha\right)}\frac{\mu_{2}^{m}}{\left(\mu_{1}^{m}\right)^{2}}+O\left(x^{-3\alpha}\right) & \leq\bar{h}_{m}\left(x\right)\label{eq:bounding h_m}\\
\leq\frac{\left(x-m\right)^{-\alpha}}{\Gamma\left(1-\alpha\right)}-\frac{\left(x-m\right)^{-2\alpha}}{\Gamma\left(1-2\alpha\right)}\frac{\mu_{2}^{m}}{\left(\mu_{1}^{m}\right)^{2}}+O\left(x^{-3\alpha}\right),\nonumber 
\end{align}
or that 
\[
\bar{h}_{m}\left(x\right)=\frac{x^{-\alpha}}{\Gamma\left(1-\alpha\right)}+g_{1}\left(x\right),
\]
where $g_{1}\sim x^{-2\alpha}$. Let $U^{m}=\Phi_{\psi}\left(\mathcal{W}_{m}\right)$,
so that $\mathbb{P}\left(U^{m}>x\right)=\bar{h}_{m}\left(x\right)=\frac{x^{-\alpha}}{\Gamma\left(1-\alpha\right)}+g_{1}\left(x\right)$.
Now, since $\mathbb{P}\left(W\geq x\right)=\frac{x^{-\alpha}}{\Gamma\left(1-\alpha\right)}+g_{2}\left(x\right)$
where $g_{2}\left(x\right)=O\left(x^{-\beta}\right)$ then $\bar{h}\left(x\right)=\mathbb{P}\left(W\geq x\right)+g_{3}\left(x\right)$
where 
\begin{equation}
g_{3}\left(x\right)=O\left(x^{-\gamma}\right),\label{eq:g_3}
\end{equation}
with $\gamma=\min\left\{ 2\alpha,\beta\right\} $. We assume our probability
space has two sequences of i.i.d r.vs $\left\{ U_{i}\right\} $ and
$\left\{ W_{i}\right\} $ where $U_{1}\sim U^{m}$ and $W_{1}\sim W$.
Applying the coupling in \lemref{annoying_lemma} with $Y=U^{m}$
and $X=W$ and substituting $g_{3}\left(x\right)$ in place of $g\left(x\right)$,
making the same calculation down to (\ref{eq:annoying lemma-2}) (note
that here $L\left(t\right)=\frac{1}{\Gamma\left(1-\alpha\right)}$)
we see that $i$-bad $I_{j}$'s can be found for $j>Ci^{\frac{\gamma+1}{\alpha+1}}$.
Using this in (\ref{eq:annoying lemma-5}), denoting $Z=\left|X-Y\right|$,
we see that the summation on the r.h.s contributes to $\mathbb{P}_{couple}\left(Z>t\right)$
at most $O\left(t^{-\gamma}\right)$, the second term contributes
$O\left(t^{-\frac{\alpha\left(\gamma+1\right)}{\alpha+1}}\right)$
whereas the last term gives not more than $O\left(t^{-\gamma\left(\frac{\gamma+1}{\alpha+1}\right)}\right)$.
We conclude that 
\[
\mathbb{P}_{couple}\left(Z>t\right)=O\left(t^{-\xi_{0}}\right),
\]
where $\xi_{0}=\frac{\alpha\left(\gamma+1\right)}{\alpha+1}$. Let
$c\left(\xi\right)=\frac{\xi-\alpha}{3\xi+\alpha}$ and note that
$c\left(\xi\right)$ is strictly increasing on $[0,\infty)$. Fix
$\xi\in[0,\xi_{0})$ and write $c:=c\left(\xi\right)$. Let $c'=\frac{3c}{\alpha}$,
$M_{1}\left(n\right)=c'n\log n$, $M_{2}\left(n\right)=c'n^{1-\alpha c'}\log n$
in (\ref{eq:main thm-5}),by Chernoff's bound we have 
\begin{alignat}{1}
\mathbb{P}_{couple}\left(\sum_{i=1}^{M_{1}\left(n\right)}n^{-\frac{1}{\alpha}}W_{i}\leq T\right) & \leq e^{sT}\left(1-n^{-1}s^{\alpha}+o\left(n^{-1}s^{\alpha}\right)\right)^{c'n\log n}\label{eq:quantitive result}\\
\mathbb{P}_{couple}\left(\sum_{i=1}^{M_{2}\left(n\right)}n^{-\frac{1}{\alpha}}W_{i}\leq n^{-c'}\right) & \leq e^{s}\left(1-n^{c'\alpha-1}s^{\alpha}+o\left(n^{c'\alpha-1}s^{\alpha}\right)\right)^{c'n^{1-\alpha c'}\log n}.\label{eq:quantitive result-1}
\end{alignat}
 taking $s=1$ in (\ref{eq:quantitive result}) and in (\ref{eq:quantitive result-1})
we see that for large enough $n$ we have 
\begin{align*}
\mathbb{P}_{couple}\left(\sum_{i=1}^{M_{1}\left(n\right)}n^{-\frac{1}{\alpha}}W_{i}\leq T\right) & \leq Cn^{-c'}\\
\mathbb{P}_{couple}\left(\sum_{i=1}^{M_{2}\left(n\right)}n^{-\frac{1}{\alpha}}W_{i}\leq n^{-c'}\right) & \leq Cn^{-c'}.
\end{align*}
 If $\left\{ Z_{i}\right\} $ are i.i.d r.vs s.t $Z_{1}\sim Z$ ,
by \cite[Eq. 1.1]{nagaev1979large} (with $t=1$) we have 
\begin{align*}
\mathbb{P}_{couple}\left(n^{-\frac{1}{\alpha}}\sum_{i=1}^{M_{1}\left(n\right)}Z_{i}>n^{-c'}\right) & \leq M_{1}\left(n\right)\mathbb{P}_{couple}\left(Z>n^{\frac{1}{\alpha}-c'}\right)\\
 & +\exp\left(1-\frac{A\left(n\right)}{n^{\left(\frac{1}{\alpha}-c'\right)}}-\log\left(\frac{n^{\left(\frac{1}{\alpha}-c'\right)}}{A\left(n\right)}\right)\right),
\end{align*}
whenever 
\begin{equation}
\frac{n^{\left(\frac{1}{\alpha}-c'\right)}}{A\left(n\right)}>1,\label{eq:condition for large deviation}
\end{equation}
 where
\[
A\left(n\right)=c'n\log n\int_{0}^{n^{\frac{1}{\alpha}-c'}}y\mathbb{P}_{couple}\left(Z\in dy\right).
\]
To see that (\ref{eq:condition for large deviation}) indeed holds,
use the fact that $\mathbb{P}_{couple}\left(Z>y\right)=O\left(y^{-\xi_{0}}\right)$
and therefore that $A\left(n\right)=O\left(cn\log\left(n\right)n^{\left(\frac{1}{\alpha}-c'\right)\left(1-\xi_{0}\right)}\right)$,
and 
\begin{align}
-c & =\left(c'-\frac{1}{\alpha}\right)\xi+1,\label{eq:c' to c}\\
 & >\left(c'-\frac{1}{\alpha}\right)\xi_{0}+1,\nonumber 
\end{align}
to conclude that for large enough $n$
\[
\frac{n^{\xi\left(\frac{1}{\alpha}-c'\right)}}{A\left(n\right)}\sim\frac{1}{\left[cn\log n\right]n^{-\xi_{0}\left(\frac{1}{\alpha}-c'\right)}}>1,
\]
 and $\frac{A\left(n\right)}{n^{\xi\left(\frac{1}{\alpha}-c'\right)}}=o\left(n^{-c}\right)$.
Using Markov's inequality we obtain for $\xi<\xi_{0}$
\begin{align}
\mathbb{P}_{couple}\left(n^{-\frac{1}{\alpha}}\sum_{i=1}^{M_{1}\left(n\right)}Z>n^{-c'}\right) & \leq n^{\left(c'-\frac{1}{\alpha}\right)\xi}nc'\log n\mathbb{E}\left(Z^{\xi}\right)\label{eq:main thm-8}\\
 & +o\left(n^{-c}\right).\nonumber 
\end{align}
By (\ref{eq:c' to c}) we see that (\ref{eq:main thm-8}) is bounded
by $Cn^{-c}\log n$. By Doob's inequality we see that
\begin{align*}
\mathbb{P}_{couple}\left(\overline{S}_{\left\{ 1,...,M_{2}\left(n\right)\right\} }>n^{-c}\right) & \leq M_{2}\left(n\right)n^{2c}n^{-1}\mathbb{E}\left(J_{1}^{2}\right)\\
 & \leq n^{-\alpha c'+2c}c\log n\mathbb{E}\left(J_{1}^{2}\right)\\
 & =n^{-c}c\log n\mathbb{E}\left(J_{1}^{2}\right).
\end{align*}
Hence, we have verified the all the inequalities in (\ref{eq:main thm-5})
with the bound $Cn^{-c}$ for $c<c\left(\xi_{0}\right)$. Following
the arguments in \ref{enu:Step_approximation} we see that if $X_{t}^{n}$
is the CTRW associated with the space-time jumps $\left\{ n^{-\frac{1}{2}}J_{i},n^{-\frac{1}{\alpha}}U_{i}\right\} $
then
\[
\rho_{J_{1}}\left(Y_{t}^{n},X_{t}^{n}\right)<Cn^{-c}.
\]
Finally, we note that since $U_{1}\in\mathfrak{L}_{s}^{s^{\alpha}}$,
by \propref{almost_every_CTRW_is_time_changed_CTRW} we may represent
$X_{t}^{n}$ as $\mathring{X}_{E_{t}^{n}}^{n}$ where $\mathring{X}_{t}^{n}$
is the CTRW associated with the space-time jumps $\left\{ n^{-\frac{1}{2}}J_{i},n^{-1}\mathring{U}_{i}\right\} $,
where $\mathring{U}_{i}\sim U^{m}$ (and so has finite mean) and $E_{t}^{n}$
is a sequence of inverse-subordinators of the subordinators $D_{t}^{n}=t+D_{\nicefrac{t}{\mu_{1}^{m}}}$
where $\mathbb{E}\left(U^{m}\right)=\mu_{1}^{m}$.
\end{enumerate}
\end{proof}
\begin{rem}
Note that the choice of the skew of the tail of $W_{1}$ need not
be $\left[\Gamma\left(1-\alpha\right)\right]^{-1}$, i.e. one can
take $W$ s.t $\mathbb{P}\left(W_{1}>t\right)\sim Ct^{-\alpha}$ for
any $C>0$. We then approximate $W_{1}$ by r.v's in $\mathfrak{L}_{s}^{C\Gamma\left(1-\alpha\right)s^{\alpha}}$.
\end{rem}
Working along the same lines of \thmref{quantitive_result} it is
not hard to see that if $Y_{t}^{n}$ is as in \thmref{quantitive_result}
but with i.i.d spatial jumps $\left\{ n^{-1}J_{i}\right\} $ where
$\mathbb{E}\left(J_{1}\right)=1$(and therefore $Y_{t}^{n}\overset{J_{1}}{\Rightarrow}E_{t}$)
then there exists a time-changed CTRW $\mathring{X}_{E_{t}^{n}}^{n}$
s.t 
\begin{equation}
\rho_{J_{1}}\left(Y_{t}^{n},\mathring{X_{E_{t}^{n}}^{n}}\right)<Cn^{-c},\label{eq:c for finite mean jump}
\end{equation}
for any $c<\frac{\xi_{0}-\alpha}{\alpha\left(\xi_{0}+1\right)}$.
Note that one can not expect for a rate of convergence in (\ref{eq:c for finite mean jump})
better then $O\left(n^{-\frac{1}{\alpha}}\right)$ as the scaling
is of $n^{-\frac{1}{\alpha}}$(unless $U_{1}\sim W_{1})$. Nevertheless,
is it possible to get arbitrarily close to $O\left(n^{-\frac{1}{\alpha}}\right)$?
Suppose we somehow manage to find a subordinator $D_{t}'$ s.t 
\[
\mathbb{P}\left(D'_{\mathcal{W}_{m}}>t\right)-\mathbb{P}\left(W_{1}>t\right)=O\left(t^{-\gamma}\right),
\]
 for $\gamma>2\alpha$. Then as $\gamma\rightarrow\infty$ we see
that (\ref{eq:c for finite mean jump}) holds for every $c<\frac{1}{\alpha}$.
Unfortunately, we could not find a way to improve $\gamma$ beyond
$2\alpha$. Another point worth mentioning is that the constant controlling
$g_{3}$ in (\ref{eq:g_3}) grows as we better approximate $W$ by
$\mathcal{W}_{m}$. This can be seen from the term $\nicefrac{\mu_{2}^{m}}{\left(\mu_{1}^{m}\right)^{2}}$
in (\ref{eq:bounding h_m}). Indeed, by the regular variation of the
tail of $W$ we see that $\nicefrac{\mu_{2}^{m}}{\left(\mu_{1}^{m}\right)^{2}}\sim m^{\alpha}$
as $m\rightarrow\infty$. An interesting question in that regard is
whether there exists a better choice of $\mathcal{W}_{m}$(possibly
where $\mathcal{W}_{m}$ has infinite slowly varying mean) so that
this undesirable phenomenon be avoided?

\subsection{Example: Pareto Distribution}

We would like to consider an example in which we use \thmref{quantitive_result}.
Let $Y_{t}^{n}$ be the CTRW associated with the i.i.d space-time
jumps $\left(n^{-\frac{1}{2}}J_{i},n^{-\frac{1}{\alpha}}W_{i}\right)$,
where $J_{1}$ has finite second moment and zero mean and $W_{1}$
has the so-called Pareto distribution $f\left(dy\right)$, i.e.
\[
\bar{f}\left(t\right)=\mathbb{P}\left(W_{1}\geq t\right)=\left\{ \begin{array}{cc}
\frac{t^{-\alpha}}{\Gamma\left(1-\alpha\right)} & t>\Gamma\left(1-\alpha\right)^{-\nicefrac{1}{\alpha}}\\
1 & t\leq\Gamma\left(1-\alpha\right)^{-\nicefrac{1}{\alpha}}
\end{array}\right.,
\]
where $0<\alpha<1$. Using \thmref{quantitive_result}, we have $\beta=\infty$
and therefore $\xi_{0}=\frac{\alpha}{7\alpha+4}$. It follows that
\[
\rho_{J_{1}}\left(Y_{t}^{n},\mathring{X}_{E_{t}^{n}}^{n}\right)<Cn^{-c},
\]
with $c<\frac{\alpha}{7\alpha+4}$, where $\mathring{X}_{E_{t}^{n}}^{n}$
is the process constructed in \thmref{quantitive_result}.
\begin{acknowledgement*}
The author would like to thank Eli Barkai and Ofer Zeitouni for  helpful
discussions on different aspects of the problem. The author would
also like to thank Mark Meerschaert and Peter Straka for reading the
manuscript and offering their insightful remarks. 
\end{acknowledgement*}
\newpage{}\bibliographystyle{plain}
\bibliography{CTRW_RWRE}

\end{document}